\documentclass[a4paper,reqno,10pt]{amsart}

\usepackage{amsmath}
\usepackage{amsthm}
\usepackage{amssymb}
\usepackage{amscd} 

\usepackage{bbm} 
\usepackage{mathrsfs} 

\usepackage{lmodern} 
\usepackage[english]{babel} 
\usepackage[utf8]{inputenc} 
\usepackage[T1]{fontenc} 
\usepackage{newtxtext,newtxmath}

\usepackage{manfnt} 

\usepackage{graphicx} 

\usepackage{comment} 

\usepackage[Lenny]{fncychap} 

\usepackage{sqrcaps} 

\usepackage{comment}

\usepackage[svgnames]{xcolor}
\usepackage{pdfcolmk}

\usepackage{tikz}
\usepackage{pifont}

\usepackage{ulem}

\swapnumbers 

\newtheoremstyle{exercise} 
  {3pt} 
  {3pt} 
  {\small\rmfamily} 
  {
} 
  {\rmfamily\scshape} 
  {.} 
  {.5em} 
  {} 

\newtheoremstyle{newplain}
  {5pt}
  {5pt}
  {\itshape}
  {}
  {\rmfamily\scshape}
  {. ---}
  {.5em}
  {}

\newtheoremstyle{newremark}
  {5pt}
  {5pt}
  {\rmfamily}
  {}
  {\rmfamily\scshape}
  {. ---}
  {.5em}
  {}

\theoremstyle{newplain}
\newtheorem*{Theorem*}{Theorem} 

\theoremstyle{newplain}
\newtheorem{Theorem}{Theorem}

\newtheorem{Corollary}[Theorem]{Corollary}
\newtheorem{Proposition}[Theorem]{Proposition}
\newtheorem{Conjecture}[Theorem]{Conjecture}
\newtheorem{Definition}[Theorem]{Definition}

\theoremstyle{newremark}
\newtheorem{Empty}[Theorem]{}

\newtheorem{Claim}[Theorem]{Claim}

\theoremstyle{exercise}

\numberwithin{Theorem}{section}
\numberwithin{Exercise}{section}

\newcommand{\R}{\mathbb{R}} 

\newcommand{\Rn}{\R^n}
\newcommand{\Z}{\mathbb{Z}} 

\newcommand{\ind}{\mathbbm{1}} 

\newcommand{\bbL}{\mathbb{L}}

\newcommand{\bbR}{\mathbb{R}}

\newcommand{\bbZ}{\mathbb{Z}}

\newcommand{\calC}{\mathscr{C}}

\newcommand{\calF}{\mathscr{F}}

\newcommand{\calI}{\mathcal I}

\newcommand{\calK}{\mathscr{K}}

\newcommand{\calM}{\mathscr{M}}
\newcommand{\calN}{\mathscr{N}}

\newcommand{\calP}{\mathscr{P}}

\newcommand{\calR}{\mathcal{R}}
\newcommand{\calS}{\mathscr{S}}

\newcommand{\bB}{\mathbf{B}}

\newcommand{\bCH}{\mathbf{CH}} 

\newcommand{\bF}{\mathbf{F}}

\newcommand{\bH}{\mathbf{H}}

\newcommand{\bI}{\mathbf{I}}

\newcommand{\bM}{\mathbf{M}}
\newcommand{\bN}{\mathbf{N}}

\newcommand{\bU}{\mathbf{U}}

\newcommand{\bZ}{\mathbf{Z}}

\newcommand{\bc}{\mathbf{c}}

\DeclareMathOperator{\rmBdry}{\mathrm{Bdry}} 
\DeclareMathOperator{\rmClos}{\mathrm{Clos}} 

\DeclareMathOperator{\rmdim}{\mathrm{dim}} 

\DeclareMathOperator{\rmid}{\mathrm{id}} 

\DeclareMathOperator{\rmLip}{\mathrm{Lip}} 

\DeclareMathOperator{\rmspt}{\mathrm{spt}} 

\newcommand{\lseg}{\boldsymbol{[}\!\boldsymbol{[}}
\newcommand{\rseg}{\boldsymbol{]}\!\boldsymbol{]}}

\newcommand{\hel} {
\hskip2.5pt{\vrule height7pt width.5pt depth0pt}
\hskip-.2pt\vbox{\hrule height.5pt width7pt depth0pt}
\, }

\def\XXint#1#2#3{{%
\setbox0=\hbox{$#1{#2#3}{\int}$}
\vcenter{\hbox{$#2#3$}}\kern-.5\wd0}}

\newcommand{\veps}{\varepsilon}

\newcommand{\la}{\langle}
\newcommand{\ra}{\rangle}

\renewcommand{\em}{\bf}

\renewcommand{\leq}{\leqslant}
\renewcommand{\geq}{\geqslant}
\renewcommand{\subset}{\subseteq}

\newlength{\drop}

\usepackage{trajan}

\begin{document}



\title[Linear isoperimetric inequality]{Linear isoperimetric inequality\\for normal and integral currents\\in compact subanalytic sets}

\def\curraddrname{{\itshape On leave of absence from}}

\author[Th. De Pauw]{Thierry De Pauw}
\address{School of Mathematical Sciences\\
Shanghai Key Laboratory of PMMP\\ 
East China Normal University\\
500 Dongchuang Road\\
Shanghai 200062\\
P.R. of China\\
and NYU-ECNU Institute of Mathematical Sciences at NYU Shanghai\\
3663 Zhongshan Road North\\
Shanghai 200062\\
China}
\curraddr{Universit\'e Paris Diderot\\ 
Sorbonne Universit\'e\\
CNRS\\ 
Institut de Math\'ematiques de Jussieu -- Paris Rive Gauche, IMJ-PRG\\
F-75013, Paris\\
France}
\email{thdepauw@math.ecnu.edu.cn,thierry.de-pauw@imj-prg.fr}

\author[R. M. Hardt]{Robert M. Hardt}
\address{Department of Mathematics\\
Rice University\\
P. O. Box 1892\\
Houston, Texas 77251}
\email{hardt@rice.edu}

\keywords{Isoperimetric inequality, normal and integral currents, subanalytic geometry, Plateau problem}

\subjclass[2010]{Primary 49Q15,49Q20,32B20,14P10; Secondary 52A40,49J45}

\thanks{The first author was partially supported by the Science and Technology Commission of Shanghai (No. 18dz2271000).
The second author was partially supported by National Science Foundation Grant DMS-1207702.}



\begin{abstract}
The isoperimetric inequality for a smooth compact Riemannian manifold $A$ provides a positive  ${\bf c}(A)$, so that for any $k+1$ dimensional integral current $S_0$ in $A$  there exists an  integral current $ S$ in $A$ with $\partial S=\partial S_0$ and $\ {\bM}(S)\leq {\bc}(A){\bM}(\partial S)^{(k+1)/k}$. Although such an inequality still holds for any compact Lipschitz neighborhood retract $A$, it may fail in case $A$ contains a single polynomial singularity.  Here, replacing $(k+1)/k$ by $1$, we  find that a linear inequality   ${\bM}(S)\leq {\bc}(A){\bM}(\partial S)$ is valid for any compact algebraic, semi-algebraic, or even subanalytic set $A$. In such a set, this linear inequality holds not only for integral currents, which have  $\Z$ coefficients, but also for normal currents having $\R$ coefficients and generally for normal flat chains with coefficients in any complete normed abelian group.  A relative version for a subanalytic pair $B\subset A$ is also true, and th
 ere are applications to variational and metric properties of subanalytic sets.

\end{abstract}

\maketitle




\tableofcontents


\section{Introduction}

Assume that $A$ is a smooth compact Riemannian manifold and that $k$ is a positive integer. The following hold.
\begin{enumerate}
\item[(A)] The singular homology group $H_k(A;\Z)$ of $A$ with integer coefficients and the homology group $\bH_k(A;\Z)$ of $A$ defined by means of integral currents are isomorphic.
\item[(B)] If an integral current $T \in \bI_k(A)$ equals $\partial S_0$ for some $S_0 \in \bI_{k+1}(A)$,  then there exists an $S \in \bI_{k+1}(A)$ such that $\partial S= T$ and $\bM(S) \leq \bc(A) \bM(T)^{(k+1)/k}$.
\item[(C)] Each homology class in $\bH_k(A;\Z)$ admits a mass minimizing representative, i.e. given $T_0 \in \bI_k(A)$ with $\partial T_0 = 0$, the following variational problem admits a minimizer:
\begin{equation*}
(\calP) \begin{cases}
\text{minimize } \bM(T)\\
\text{among } T \in \bI_k(A) \text{ with } T - T_0 = \partial S \text{ for some } S \in \bI_{k+1}(A).
\end{cases}
\end{equation*}
\end{enumerate}
\par 
These have been established by \textsc{H. Federer and W.H. Fleming}, \cite[6.3]{FED.FLE.60}. Note that a smooth isometric embedding $A \subset \Rn$ of the manifold into some Euclidean space exhibits $A$ as a Lipschitz neighborhood retract. This means there exists an open neighborhood $U$ of $A$ in $\Rn$ and a Lipschitzian retraction $f : U \to A$ onto $A$ (one can choose $f$ to have the same class of smoothness as $A$). It ensues that $\bH_k(A;\Z)$ and $\bH_k(U;\Z)$ are isomorphic, and conclusion (A) now follows from the
 the deformation theorem,  \cite[4.4.2]{GMT} which shows that each integral cycle $T \in \bI_k(U)$ is homologous to some polyhedral cycle $P \in \bI_k(U)$, of comparable mass, coming from a fixed complex.  Conclusion (B) also follows from a careful application of this deformation theorem. In order to establish (C), one considers a mass minimizing sequence $\la T_j \ra_j$ for $(\calP)$. According to (B) there are $S_j \in \bI_{k+1}(A)$ such that $\partial S_j = T_j - T_0$ and
\begin{equation*}
\bM(S_j) \leq \bc(n) \bM(T_j-T_0)^{(k+1)/k} \leq \bc\left(n,\bM(T_0)\right) 
\end{equation*}
if $j$ is large enough. Referring to the compactness theorem of integral currents, corresponding subsequences of $\la T_j \ra_j$ and $\la S_j \ra_j$ converge  in flat norm to, respectively $T \in \bI_k(A)$ and $S \in \bI_{k+1}(A)$ such that $\partial S = T - T_0$\ , thus $T$ and $T_0$ are homologous in $A$. As $\bM$ is lower semicontinuous with respect to convergence in the flat norm, $T$ minimizes mass in its homology class. 
\par 
In this paper we study these questions with $A$  being a compact subanalytic subset of $\Rn$, and $\bI_k(A) = \bI_k(\Rn) \cap \{ T : \rmspt T \subset A \}$. Whereas {\bf semi-algebraic sets} \cite{BOCHNAK.COSTE.ROY} are defined by finitely many polynomials, the larger class of {\bf subanalytic sets} \cite{BIE.MIL.88} includes sets defined locally by real analytic functions as well as their images under proper real analytic maps.

Such sets may fail to be Lipschitz neighborhood retracts, the methods of \cite{FED.FLE.60} do not apply, and the isoperimetric inequality of (B) may in fact fail: 

 For $N=2,3,\ldots$, we define the semi-algebraic set
\begin{equation*}
A_N = \R^3 \cap \left\{ (x,y,z) : z^{2N} = x^2 + y^2 \text{ and } 0 \leq z \leq 1 \right\},
\end{equation*}
i.e. $A_N$ is obtained from the rotation around the $z$ axis of the graph of $z = x^{1/N}$, $0 \leq x \leq 1$; it has a cusp at the origin. Given $0 < h < 1$ we consider $T_h \in \bI_1(A_N)$ an oriented circle of multiplicity 1 on $A_N$, at height $h$. It is not hard to show that there exists a unique $S_h \in \bI_2(A_N)$ such that $\partial S_h = T_h$, and
\begin{equation*}
\bM(S_h) = 2\pi \int_0^h z^N \sqrt{1 + (N z^{N-1})^2}dz \,.
\end{equation*}
Since $\bM(T_h) = 2\pi r = 2 \pi h^N$, where $r$ is the radius of the circle $A_N \cap \{ z = h \}$, we infer the following. For each $q > 1$, choosing $N$ so large that $q = \frac{N+1}{N} + \veps$ for some $\veps > 0$,
\begin{equation*}
\lim_{h \to 0^+} \frac{\bM(S_h)}{\bM(T_h)^q} = \lim_{h \to 0^+} \frac{2\pi \left( \frac{h^{N+1}}{N+1}\right)}{\left( 2\pi h^N \right)^q} = \lim_{h \to 0^+} \frac{(2\pi)^{1-q}}{N+1}\left( \frac{1}{h}\right)^{N\veps} = \infty \,.
\end{equation*}
This shows that in conclusion (B) above, we cannot hope for an inequality $\bM(S) \leq \bc(A) \bM(T)^q$ with an exponent $q > 1$ depending only on the dimension of $T$ and the constant $\bc(A)$ depending only on the semialgebraic set $A$. In fact if $q$ is allowed to depend only on the dimension of $T$, then $q=1$ is the only possible choice, as illustrated by these simple calculations.  Incidentally the computations also show that $A_N$ is not a Lipschitz neighborhood retract: If there were a Lipschitzian retraction $f : U \to A_N$, considering $D_h \in \bI_2(\R^3)$ the unique flat disk in $\R^3$ with $\partial D_h = T_h$, and $h$ small enough for $\rmspt D_h \subset U$, we would have $f_\#D_h = S_h$ and in turn
\begin{equation*}
\frac{2 \pi h^{N+1}}{N+1} \leq \bM(S_h) = \bM(f_\#D_h) \leq \pi(\rmLip f)^2 h^{2N},
\end{equation*}
a contradiction as $h \to 0^+$.
\par 
Our main result is as follows.
\begin{Theorem*}
Let $A \subset \Rn$ be a compact subanalytic set. There exists $\bc(A) > 0$ with the following property. For every $k=0,1,2,\ldots$ and every $S_0 \in \bI_{k+1}(A)$, there exists $S \in \bI_{k+1}(A)$ such that $\ \partial S=\partial S_0\ $ and $\ \bM(S) \leq \bc(A) \bM(\partial S)$. 
\end{Theorem*}
\par 
This is known as a {\it linear} isoperimetric inequality because of the absence of an exponent in the righthand term. Moreover, the linear inequality holds, via the same proof, with $\Z$ replaced by any normed, complete Abelian group $G$ of coefficients. When $G=\R$, with the absolute value norm, the fact that both $\bM(\lambda S)=\lambda \bM(S)$ and  $\bM(\lambda \partial S)=\lambda \bM(\partial S)$ for arbitrarily small positive $\lambda$ shows that both sides of a valid isoperimetric inequality must be homogeneous of the same degree, which can therefore be set equal to 1.
\par 
Our proof is based on two facts of metric nature, regarding a compact subanalytic set $A$. A basic topological property is that $A$ is triangulable \cite{HAR.76}, \cite{HIR.75}, i.e. there exists a simplicial complex ${\calK}$ and a subanalytic homeomorphism $\phi : |{\calK}| \to A$. For $x \in A$, there exists $r(x) > 0$ such that the intersections with open balls  $|{\calK}| \cap \bU(\phi^{-1}(x),r)$ are contractible for $0 < r \leq r(x)$, and it ensues that $A$ is locally contractible. Here we will use a strengthening due to \textsc{G. Valette}, \cite{VAL.12} stating that the local homotopy $h :[0,1] \times (A \cap U) \to A \cap U$ from the identity to a constant can be chosen to be Lipschitz. This already implies $A$ is ``locally acyclic'' with respect to the chain complex of integral currents and in turn, that $H_k(A;\Z)$ and $\bH_k(A;\Z)$ are isomorphic, see \cite{DEP.06.HOM}, which is analogous to (A) above.  See also \cite{FUN.16}.  A second consequence of triangulabi
 lity is that if $x \in A$ and $0 < s \leq s(x)$ is small enough, then, for  $s/2 < t < t' < s$, the spherical links $A \cap \rmBdry \bU(x,t)$ and $A \cap \rmBdry \bU(x,t')$ are subanalytically  homeomorphic.  Here we use the work of \textsc{A. Parusi\'nski}, \cite{PAR.94} on locally Lipschitz stratification of subanalytic sets to find, away from the origin,  intervals of radii for links which are uniformly bilipschitz equivalent.
\par 
After our Main Theorem and the analog of (B) above are established, the analog of (C) follows along the same lines sketched at the beginning of this introduction. It suffices to observe that for proving (B), the particular power $(k+1)/k$ plays no significant role. An existence theorem for the Plateau problem in semialgebraic sets has been obtained recently by \textsc{Q. Funk} using different methods, \cite{FUN.16}.
\par
As discussed in \S 5 , the linear isoperimetric inequality for compact subanalytic sets is also valid for normal currents with $\bbR$ coefficients or even normal flat chains with coefficients in any complete normed abelian group. It plays a role in relating not only various homology theories but also a duality between  a homology based on normal currents and a cohomology based on normal cochains called charges \cite{DEP.HAR.PFE.17}.
\par
 A large number of geometric variational problems involving support constraints, boundary constraints, or free boundaries can be formulated with various groups of chains, cycles, boundaries, or homologies. Generalizing (C), we give in \S \ref{59}, the existence theory for one problem related to minimizing mass in a relative homology class of a pair  $B\subset A$ of compact subanalytic sets. This uses  a relative isoperimetric inequality which bounds the the part of the mass in $A\setminus B$ of a suitable chain in terms of the part of the mass in $A\setminus B$ of its boundary. There remain many interesting regularity questions concerning solutions of this and other variational problems in semi-algebraic and subanalytic sets.
\par 
 Interesting inequalities concerning functions defined on  singular algebraic varieties or subanalytic domains are found in works of {\sc L. Bos} and {\sc P. Milman} \cite{BOS.MIL.95}, \cite{BOS.MIL.06} and {\sc A.Valette} and {\sc G. Valette} \cite{VAL.VAL.20} .
\par
Our notations regarding integral and normal currents, and flat chains are consistent with that of \cite{GMT}.

\section{Bilipschitz Equivalence}
\begin{Empty}[Linear isoperimetric inequality]
Let $k=0,1,2,\ldots$ and $A \subset \Rn$. We say that $A$ {\em satisfies the linear isoperimetric inequality of dimension $k$} whenever the following holds. There exists $\bc(A,k) > 0$ such that for each $S_0 \in \bI_{k+1}(A)$ there exists $S \in \bI_{k+1}(A)$ with $\partial S=\partial S_0$ and $\bM(S) \leq \bc(A,k) \bM(\partial S)$.
\end{Empty}

\begin{Empty}[Bilipschitz equivalence]
\label{32}
Let $X , Y \subset \Rn$ and $\bbL > 0$. 
We say that a bijective map $\phi : X \to Y$ is {\em $\bbL$-bilipschitz} if $\phi$ is Lipschitz, $\phi^{-1}$ is Lipschitz and $\max \{ \rmLip \phi , \rmLip \phi^{-1} \} \leq \bbL$.
In case such map $\phi$ exists, we say that $X$ and $Y$ are {\em bilipschitz equivalent}.
\end{Empty}

\begin{Proposition}
\label{33}
Let $k=0,1,2,\ldots$, and $X,Y \subset \Rn$.
Assume that $X$ and $Y$ are bilipschitz equivalent.
It follows that $X$ satisfies the linear isoperimetric inequality of dimension $k$ if and only if $Y$ does.
\end{Proposition}

\begin{proof}
Let $\phi : X \to Y$ be an $\bbL$-bilipschitz homeomorphism. 
One infers from Kirszbraun's Theorem \cite[2.10.43]{GMT} that $\phi, \phi^{-1}$ admit extensions $f, g : \Rn \to \Rn$ with $\max \{ \rmLip f , \rmLip g \} \leq \bbL$.
Assume $X$ satisfies the linear isoperimetric inequality of dimension $k$ with constant $\bc > 0$.
Let $S_0 \in \bI_{k+1}(Y)$.
Define $S_0' = g_\#S_0 \in \bI_{k+1}(\Rn)$.
Notice that $\rmspt(S_0') \subset X$, \cite[4.1.14 p. 371]{GMT}.
Thus there exists $S' \in \bI_{k+1}(X)$ such that $\partial S'=\partial S'_0$ and $\bM(S') \leq \bc \bM(\partial S')$.
Define $S = f_\#S'$ and notice as before that $S \in \bI_{k+1}(Y)$.
Further note that $f_\# g_\# S = (f \circ g)_\# S = S$ where the last equality follows from \cite[4.1.15 p. 372]{GMT} and the fact that $f \circ g = \rmid_{\Rn}$ on $\rmspt S$.
Finally,
\begin{multline*}
\bM(S) = \bM(f_\#g_\#S) \leq (\rmLip f)^{k+1} \bM(S') \leq (\rmLip f)^{k+1} \bc \bM(\partial S')\\ = (\rmLip f)^{k+1} \bc \bM(g_\# \partial S)  \leq (\rmLip f)^{k+1} (\rmLip g)^{k} \bc \bM(\partial S) \leq \bc {\bbL}^{2k+1} \bM(\partial S).
\end{multline*}
\end{proof}

\section{Two Properties of Subanalytic Sets}

A finitely triangulated compact space is trivially locally contractible at each point. 
For a subanalytic set $A$, this local contraction may be chosen to be {\it Lipschitz} according to the following result of \textsc{G. Valette}, \cite[Theorem 2.3.1]{VAL.12} even though $A$ itself may fail to be a Lipschitz neighborhood retract.

\begin{Theorem}
\label{thm.1}
Any point $a$ in a closed subanalytic subset $A$ of $\Rn$ has a compact subanalytic neighborhood $K \subset A$ and a Lipschitz deformation contraction $h:[0,1] \times K\to K$ so that 
$h(0,x)=x$ and $h(1,x)=a$ for $x\in K$ and $h(t,a)=a$ for $t\in [0,1]$.
\end{Theorem}

From this we obtain at once a local version of the linear isoperimetric inequality.

\begin{Corollary}
\label{cor.1}
Suppose that  $J \in {\bI}_{k}(K)$ and that $\partial J = 0$ in case $k>0$. Then the chain  $H:=-h_\#\left( \lseg 0,1 \rseg \times J \right)\ $ belongs to ${\bI}_{k+1}(K)\ $ and satisfies
$$
 \partial H\ =\  \begin{cases} J & {\rm if}\ \ $k\ >\ 0$\\
                                              J - J(1)\lseg a\rseg  &{\rm if}\ \ k=0 \end{cases}        
  \quad\quad{\rm  and} \quad\quad{\bM}(H) \leq ( \rmLip h)^{k+1} {\bM}(J)\ .
$$
\end{Corollary}             

\begin{proof}
Clearly $H\in {\calR}_{k+1}(K)$ because $h$ is Lipschitz, and the homotopy formula \cite[4.1.9]{GMT} shows that, for $k>0$,
\begin{align*}
-\partial H\ =\ h_\#\partial ( \lseg 0,1 \rseg\times J )\ 
&= h_\#( \lseg 1 \rseg\times J)  - h_\#( \lseg 0 \rseg \times J) - h_\#( \lseg 0,1 \rseg\times \partial J )\\
&= \quad \quad \; 0\quad\quad  -\quad \quad  J\quad\quad\, -\quad 0\quad  \in {\calR}_{k}(K)\ ,
\end{align*}
Here
$h(0,\ \cdot\ ) = {\rm id}$ and $h_\#( \lseg 1 \rseg\times J)=0\ $,
by \cite[4.1.20]{GMT}, because $\rmspt h_\#( \lseg 1 \rseg\times J) \subset \{a\}$ and $h_\#( \lseg 1 \rseg\times J) \in \bI_{k}(\Rn)$ with $k \geq 1$. In case $k=0$,
\cite[4.1.9]{GMT} shows that
$$
-\partial H\ =\ h_\#( \lseg 1 \rseg\times J)  - h_\#( \lseg 0 \rseg \times J)\ =\  J(1)\lseg a\rseg\ -\ J
$$
because.  $h_\#( \lseg 1 \rseg\times J)(f)=J(f(a))=f(a)J(1)$ for $f\in {\calC}^\infty_c(\Rn)$.
In either case, we easily estimate
$$
\bM(H)\ \leq ( \rmLip h)^{k+1} {\bM}\left(\lseg 0,1 \rseg \times J\right)\ =\ ( \rmLip h)^{k+1} {\bM}(J)\ .
$$ 
\end{proof}

To use this Corollary in the proof of the Main Theorem via local modifications of the given current $S$, one needs to partition $S$ into finitely many small pieces each contained in a Lipschitz contractible piece of $A$. The choice of the partition necessarily depends on the current $S$. Since one still needs the mass bounds in all constructions to be independent of this choice, we first verify some bilipschitz equivalences.

\begin{Empty}[Finite unions of links]
Any $a\in A$ and $r>0$ determine a {\em spherical link} $L_{r}^{a} := A\cap\rmBdry \bU(a,r)$. Also for $\ell$-tuples ${\vec a}=( a_1,\dots ,a_\ell )\in A^\ell\ $ and ${\vec r}=( r_1,\dots ,r_\ell )\in (\R_+^*)^\ell$ we may denote the corresponding union of links  $ L_{\vec r}^{\vec a} := \cup_{i=1}^\ell  L_{r_i}^{a_i}$. 
\end{Empty}

The next result follows from work of \textsc{A. Parusi\'nski}, \cite{PAR.94} on  locally Lipschitz trivial stratification of subanalytic sets. To indicate individual rescalings of the $r_i$, we use, for ${\vec \lambda}=( \lambda_1,\dots ,\lambda_\ell )$ and $\vec r=( r_1,\dots ,r_\ell )$, the notation $\vec\lambda\vec r = (\lambda_1 r_1,\dots ,\lambda_\ell r_\ell)$.

\begin{Theorem}[Link Bilipschitz Equivalence]
\label{thm.2}
Associated to each pair  $\vec a,\ \vec r$ of such $\ell$-tuples are numbers $\frac 12 <\lambda_{\vec r, i}^{\vec a} <\mu_{\vec r, i}^{\vec a} < 1$ so that, for  any $\lambda_i,\mu_i\in [\lambda_{\vec r, i}^{\vec a}, \mu_{\vec r, i}^{\vec a}]$, the two corresponding link unions $ { L}_{{\vec\lambda} {\vec r}}^{\vec a}$ and ${ L}_{{\vec\mu} {\vec r}}^{\vec a}$  are uniformly bilipschitz equivalent 
 by a map sending  ${ L}_{\lambda_i r_i}^{a_i }$ to $ { L}_{ \mu_i r_i}^{a_i}$\quad for $i=1,\dots ,\ell$. 
\end{Theorem}

\begin{proof}
We will work in $\R^\ell\times\Rn$.  Here each set
$$
L_i\ :=\ \bigcup_{\vec\lambda\in [\frac 12,1]^\ell }\{\vec\lambda\} \times L_{\lambda_i r_i}^{a_i }\ 
=\ \left\{ (\vec\lambda,x)\ :\ \vec\lambda\in [\frac 12,1]^\ell ,\ x\in A,\  |x-a_i|^2=\lambda_i^2r_i^2\ \right\}
$$ 
is clearly subanalytic. Let $L=\cup_{i=1}^\ell L_i=  \cup_{\vec\lambda\in [\frac 12,1]^\ell }\{\vec\lambda\} \times L_{\vec\lambda{\vec  r}}^{\vec a}$. By \cite[Theorem 1.4]{PAR.94}, there exists a locally Lipschitz trivial stratification ${\calS}$ of $ L$ having subanalytic strata and being compatible with every $ L_i$. This means that each $S\in{\calS}$ which intersects any $ L_i$ is contained in $ L_i$. Thus, the subfamily $ \{ S\in{\calS}\,:\,S\cap  L_i\neq\emptyset \}$ forms a stratification of $ L_i$ .  Let
$$
p(\vec\lambda , x)\ =\ \vec\lambda\quad{\rm and} \quad p_i(\vec\lambda , x)\ =\ \lambda_i\quad{\rm for\ each}\quad  i\in \{ 1,\dots ,\ell\}\ .
$$
For each $S\in{\calS}$,   the critical set of $p|S$,\  $Z_{S}: = S \cap \{ x:  {\rm dim}p( {\rm Tan} (S,x) ) <\ell\}$, is subanalytic. Hence, $E={\rm Bdry}[0,1]^\ell\cup\cup_{S\in{\calS}}p(Z_S)$  is a compact subanalytic set of dimension $\leq \ell -1$. Similarly the critical set of $p_i|S$,\  $Z_{S,i}: = S \cap \{ x:  p_i( {\rm Tan} (S,x) )=0\}$, is subanalytic. And  $E_i=\{ 0\} \cup\{ 1\} \cup p_i(Z_{S,i})$, being  compact and subanalytic of dimension $0$, is finite.  
\par
To prove the theorem we choose, for each $i\in\{ 1,\dots ,\ell \}$,   a nontrivial closed subinterval  $[\lambda_{\vec r, i}^{\vec a},\ \mu_{\vec r, i}^{\vec a} ]$ of one of the components $I_i$ of $[\frac 12,1]\setminus E_i$ so that the corresponding rectangle  
$$
R_{\vec r}^{\vec a}\ :=\ \prod_{i=1}^\ell [\lambda_{\vec r, i}^{\vec a},\ \mu_{\vec r, i}^{\vec a} ]\quad\subset\quad \big[\,\frac 12\, ,1\big]^\ell\setminus E\ .
$$
To obtain uniformly bilipschitz maps between $L\cap p^{-1}\{\vec\lambda\}$ and $L\cap p^{-1}\{\vec\mu\}$ for $\vec\lambda,\vec\mu\in R_{\vec r}^{\vec a}$, we repeat arguments of \cite{PAR.94} on lifting Lipschitz vectorfields to locally Lipschitz stratifications.  Specifically, as in the proof of \cite[Theorem 1.6, Lemma 1.7]{PAR.94}, the constant unit vectorfield ${\mathbf e}_i$ on $\prod_{j=1}^\ell I_j$ can be lifted via $p$ to a locally Lipschitz vectorfield  $X_i$ on $ L\cap p_i^{-1}(I_i)$. Moreover, each of the integral curves of $X_i$ stays in the stratum $S\in{\calS}$ where it begins.
\par
To go between two points $\vec\lambda,\vec\mu$ in $R_{\vec r}^{\vec a}$, we use intervals $J_1,\dots ,J_\ell$ going between the intermediary points
$$
\vec\nu_0=\vec\mu\ ,\quad \vec\nu_i=(\mu_1, \dots ,\mu_{i}, \lambda_{i+1},\dots,\lambda_\ell)\quad {\rm for}\ i=1,\dots ,\ell-1,\quad \vec\nu_\ell =\vec\lambda\ . 
$$
Thus on the interval $J_i$ which goes from $\vec\nu_{i-1}$ to $\vec\nu_i$ only the $i$th coordinate changes and the tangent is ${\mathbf e}_i$. 
Matching the endpoints of the integral curves of the lift of ${\mathbf e}_i$ lying above $J_i$ gives a bilipschitz equivalence between $ \{ {\vec\nu_{i-1}}\}\times { L}_{\vec\nu_{i-1}\vec r}^{\vec a}$ and $ \{ \vec\nu_{i}\}\times { L}_{\vec\nu_{i}\vec r}^{\vec a}$. The bilipschitz constants are bounded independent of the initial point $\vec\nu_{i-1} \in R_{\vec r}^{\vec a}$ as in \cite[Theorem 1.6, Lemma 1.7]{PAR.94}. Taking the composition of these $\ell$ bilipschitz equivalences finally gives the desired uniformly bilipschitz equivlence between $L\cap p^{-1}\{\vec\lambda\}$ and $L\cap p^{-1}\{\vec\mu\}$. Moreover, for 
each individual  $i$, $ \{ \lambda\}\times { L}_{\lambda_i r_i}^{a_i}\to\{\mu\}\times{ L}_{\mu_i r_i}^{a_i}$ because each integral curve stays completely in only one stratum $S\in{\calS}$, and each such stratum lies in at most one $L_i$.
\end{proof}

\section{Proof of the Theorem}

In this section we will use $\bc$ (rather than $\bc(A)$) to denote a constant depending only on $A$. Its value may increase in the course of the proof, even in a single chain of inequalities.
Since dividing  $A$ into its finitely many path components decomposes both $S_0$ and $\partial S_0$, we may assume that $A$ itself is path-connected.   
\subsection{Partition into contractible regions}
\par
By compactness of $A$ and Theorem \ref{thm.1}, we may choose a finite family of open balls 
$$
\{\ \bU(a_i\, ,\, r_i/2)\ :\ i=1,2,\dots\,  ,\ell\ \}
$$ 
covering $A$ where the $a_i$ are distinct points of $A$ and each $\bU(a_i,r_i)\cap A$ lies in a compact neighborhood $K_i \subset A$ of $a_i$ having a Lipschitz deformation contraction $h_i$ to $\{a_i\}$, 
\par
We would like to make, for some fixed $\vec\lambda\in [1/2,1]^\ell$,  separate adjustments of $S_0$ in each of the open balls $\ U_i^{\lambda_i} := \bU(a_i,\lambda_i r_i)\ .$ Since the $U_i^{\lambda_i}$ will likely overlap, we will work with the corresponding disjoint open sets
$$
W_1^{\vec\lambda} := U_1^{\lambda_1},\quad  W_2^{\vec\lambda} := U_2^{\lambda_2}\setminus \rmClos U_1^{\lambda_1},\quad \ldots\ ,\quad W_\ell^{\vec\lambda} := U_\ell^{\lambda_\ell}\setminus\bigcup_{h=1}^{\ell -1}\rmClos U_i^{\lambda_h }\ ,
$$
noting that $A\subset \cup_{i=1}^\ell\rmClos W_i^{\vec\lambda}$ and $\rmBdry W_i^{\vec\lambda}  \subset \cup_{h=1}^i\rmBdry U_h^{\lambda_h }$, a finite union of $n-1$ spheres. 

We will eventually  choose the radii $\lambda_1,\lambda_2,\dots ,\lambda_\ell$ in order  which determine the sets $W_1^{\vec\lambda} , W_2^{\vec\lambda} ,\dots ,W_\ell^{\vec\lambda} $ (because  $W_i^{\vec\lambda}$ only depends on  $\lambda_1,\dots , \lambda_i$\ ).
\par
In choosing the radii $\lambda_i$, there are some Lebesgue null sets of $\lambda$ that one must avoid. To describe one, first, note  that, for any finite Radon measure $\alpha$ on $\bbR^n$, the set
$$
\Lambda_\alpha :=  \{ \lambda > 0\ :\ \alpha( L_{\lambda r_i}^{a_i}) > 0\ {\rm for\ some}\ i\ \}\  
$$
is at most countable because  $ \{ \lambda :\ \alpha( L_{\lambda r_i}^{a_i}) > 1/j \}$ is finite for each $j\in{\mathbb N}$.   One useful consequences, for any positive $\lambda_1,\dots ,\lambda_i\not\in\Lambda_{\| S_0\| +\|\partial S_0\|}$, is that 
\begin{equation*}
(\| S_0\|+\|\partial S_0\|)\left({\rm Bdry} (W_i^{\vec\lambda})\right)= 0\ ,\quad
 S_0\hel\overline{W}_i^{\vec\lambda}=S_0\hel W_i^{\vec\lambda}\ ,\quad   (\partial S_0)\hel\overline{W}_i^{\vec\lambda}=(\partial S_0)\hel W_i^{\vec\lambda} ,
\end{equation*} 
for $i=1,\dots , \ell$.  The disjoint open sets $W_i^{\vec\lambda}$ now give the two decompositions
\begin{equation} 
\label{decomp}
S_0 \ =\  \sum_{i=1}^\ell S_0\hel W_i^{\vec\lambda}\quad{\rm and}\quad (\partial S_0)\  =\  \sum_{i=1}^\ell (\partial S_0)\hel W_i^{\vec\lambda}\ ,
\end{equation}
obtained by ignoring the boundaries of the $W_i^{\vec\lambda}$.

\subsection{Proof of the case $k=0$}
\begin{proof}

We will make essentially two different applications of  Corollary \ref{cor.1} with $k=0$. 
First, for $\vec\lambda =(\lambda_1,\dots ,\lambda_\ell)$  as above and  $i\in\{ 1,2,\dots ,\ell\}$, we apply Corollary  \ref{cor.1}  with
$J= (\partial S_0)\hel W_i^{\vec\lambda}\ $ to find that  
$$
E_i\ :=\ -h_{i\#} \left(\lseg 0, 1\rseg\times [(\partial S_0)\hel W_i^{\vec\lambda}]\right)
\in \bI_1(A)
$$ 
 has 
$$
\partial E_i =  (\partial S_0)\hel W_i^{\vec\lambda}\ -\ g_i\lseg a_i\rseg\quad{\rm and}\quad \bM(E_i)\leq \rmLip(h_i)\bM[(\partial S_0)\hel W_i^{\vec\lambda}]\ ,
$$ 
where $g_i = (\partial S_0)\left(\ind_{W_i^{\vec\lambda}}\right)\in \bbZ$  is the total multiplicity of $\partial S_0$ in $W_i^{\vec\lambda}$.  Observe that by (\ref{decomp}),
\begin{equation}
\label{tot.mult}
\sum_{i=1}^\ell g_i\ =\ \sum_{i=1}^\ell[ (\partial S_0)\hel W_i^{\vec\lambda})](1)\ = \left(\sum_{i=1}^\ell (\partial S_0)\hel W_i^{\vec\lambda}\right)(1)\ =\ (\partial S_0)(1)\ =\ 0\ .
\end{equation}

Second, we can use  Corollary \ref{cor.1} to bound the intrinsic diameter of $A$. Consider the simple situation in  Corollary  \ref{cor.1} when 
$J$ is a single point with multiplicity, say $J=g\lseg b,\rseg$ where $g\in\bbZ$ and $b\in K$. The resulting  $ H =  -h_{\#} \left(\lseg 0, 1\rseg\times g\lseg b\rseg\right)$ is $-g$ times the path  $t\mapsto h(t,b)$ from $b$ to $a$ and
 $$
 \partial H\ =\ g\lseg b\rseg - g\lseg a\rseg\quad{\rm and}\quad \bM (H)\ =\ |g|{\rm length}\left(h(\cdot ,b)\right)\ \leq \ |g|\rmLip h\ .
 $$
 Thus the intrinsic diameter of $K$ is $\leq 2\times {\rm length}\left(h(\cdot ,b)\right)\leq 2\rmLip h$. Inasmuch as the intrinsic diameter of the union of two intersecting sets is at most the sum of their intrinsic diameters, we readily deduce that the intrinsic diameter of $A$ is at most
 $D(A) :=2\sum_{i=1}^\ell \rmLip(h_i)$. 
 
 To complete the proof of the $k=0$ case, we choose, for each $i=1,\dots ,\ell$, a curve $\gamma_i:[0,1] \to A\ $ from $a_1$ to $a_i$ with ${\rm length}(\gamma_i)\leq D(A)$. We now define 
 $$
 S\ :=\ \sum_{i=1}^\ell E_i + g_i\gamma_{i\#}\lseg 0,1\rseg
 $$
 and verify, by (\ref{tot.mult}), that  
\begin{align*}
 \partial S\ &=\ \sum_{i=1}^\ell \left( (\partial S_0)\hel W_i^{\vec\lambda}\ -\ g_i\lseg a_i\rseg\right)\ +\ \left(g_i \lseg a_i \rseg - g_i \lseg a_1 \rseg \right)\\
 &=\ \partial S_0\ -\ \left( \sum_{i=1}^\ell g_i\right)\cdot\lseg a_1\rseg\ =\ \partial S_0\ -\ 0\ ,
\end{align*}
 and
\begin{align*}
 \bM(S)\ &\leq\ \sum_{i=1}^\ell \bM(E_i) + \sum_{i=1}^\ell |g_i| D(A) \\
 &\leq \ \sum_{i=1}^\ell \rmLip(h_i)\bM[(\partial S_0)\hel W_i^{\vec\lambda}] \ +\ D(A)\bM[(\partial S_0)\hel W_i^{\vec\lambda}]\ \leq \bc\bM(\partial S_0)\ .
\end{align*}
 \end{proof}
 
 \subsection{Proof of the case $k\geq 1$}
 
 \begin{proof}

 We use induction on $\rmdim A$.  For $\rmdim A=0$, the current  $S_0\in\bI_{k+1}(A)$ necessarily vanishes, and the theorem is trivially true. Also for $\rmdim A=1$, $S_0\in\bI_{k+1}(A)$ is nonvanishing only for $k=0$, in which case the theorem was established in the previous section \S 4.2. So we now assume that $k\geq 1$, that $\rmdim A \geq 2$ and inductively,  that the theorem is true for any compact subanalytic set of dimension less than $\rmdim A$.  
\par 
 For this inductive step,  note that the links  $ L_{\lambda_i r_i}^{a_i}= A\cap \rmBdry U_i^{\lambda_i}$, as well as all  the sets $A\cap \rmBdry W_i^{\vec\lambda}$, are compact and subanalytic of dimension $< \dim A$.  We need some more discussion about the choice of the radii $\lambda_i$ , the corresponding balls $U_i^{\lambda_i}$, disjoint open sets $W_i^{\vec\lambda}$ and their boundaries. First the distinctness  of the centers of the $U_i^{\lambda_i}$  already guarantees the pairwise transversality of their spherical boundaries. 
 \par
We  describe a particular subdivision of $\bigcup_{i=1}^\ell\rmBdry W_i^{\vec\lambda}$ by using the sequence of closed spherical regions 
$$
\Gamma_i^{\vec\lambda} := (\rmBdry U_i^{\lambda_i}) \cap (\rmBdry W_i^{\vec\lambda})\quad =\quad  (\rmBdry U_i^{\lambda_i}) \setminus \bigcup_{h=1}^{i-1} U_h^{\lambda_h}\ ,
$$
whose interiors $\ \Gamma_i^{\vec\lambda\,\circ}\ $are disjoint. While each  $\Gamma_i^{\vec\lambda}$ clearly does not overlap any $W_h^{\vec\lambda}$ for $h\leq i$, we will also be interested in the cover of $\ \Gamma_i^{\vec\lambda}\ $ by the closed $n-1$ dimensional spherical regions 
$$
\Gamma_{i,j}^{\vec\lambda} := \Gamma_i^{\vec\lambda} \cap \rmClos W_j^{\vec\lambda} \quad{\rm for} \quad j=i+1,\dots , \ell\ ,
$$
whose interiors $\Gamma_{i,j}^{{\vec\lambda}\,\circ}$ are disjoint and whose boundaries are  contained in $n-2$ dimensional spheres. Note that
$$
A \cap \rmBdry W_i^{\vec\lambda}\quad =\quad A \cap  \left(\Gamma_i^{\vec\lambda}\ \cup\ \bigcup_{h=1}^{i-1} \Gamma_{h,i}^{\vec\lambda}\right)\quad 
=\quad A \cap \left(\bigcup_{h=1}^{i-1} \Gamma_{h,i}^{\vec\lambda}\ \cup\ \bigcup_{j=i+1}^{\ell} \Gamma_{i,j}^{\vec\lambda}\right) \ .
$$
Again all the corresponding interior spherical regions, $\Gamma_{h,i}^{{\vec\lambda}\,\circ}$ or $\Gamma_{i,j}^{{\vec\lambda}\, \circ}$,  are disjoint for each fixed $i$, with $h$ ranging in $\{1,\dots ,i-1\}$ \ and $j$ ranging in $\{i+1,\dots ,\ell\}$. 
\par
Many of the partitioning domains $W_i^{\vec\lambda}$ or many of their boundary regions $\Gamma_{i,j}^{{\vec\lambda}\,\circ}$ may be empty. In any case,  for each nonempty $W_i^{\vec\lambda}$,  the boundary $\rmBdry W_i^{\vec\lambda}$ has a single nonempty ``outward-pointing'' region $ \Gamma_i^{\vec\lambda}\subset\rmBdry U_i^{\lambda_i}$, which may decompose into various $\Gamma_{i,j}^{\vec\lambda} $ for some later $j>i$. And any remainder $\rmBdry W_i^{\vec\lambda}\setminus  \Gamma_i^{\vec\lambda}$ then consists of ``inward-pointing'' regions $\Gamma_{h,i}^{\vec\lambda}$ coming from distinct spheres for some earlier $h<i$.
\par
Each set $A\cap \Gamma_i^{\vec\lambda}\ $ is contained in the link union $ L_{\vec\lambda\vec r}^{\vec a}$.
Moreover, $A\cap \Gamma_i^{\vec\lambda}$ is mapped to $A\cap \Gamma_i^{\vec\mu}$ under the bilipschitz equivalence that maps $ L_{\vec\lambda\vec r}^{\vec a}$ to $ L_{\vec\mu\vec r}^{\vec a}$, obtained in Theorem \ref{thm.2}.  
In fact, since each individual  link $L_{\lambda_ir_i}^{ a_i} = A\cap\rmBdry U_i^{\lambda_i} $ is mapped to the corresponding link $ L_{\mu_ir_i}^{ a_i} = A\cap\rmBdry U_i^{\mu_i }$,   a similar $\vec\lambda\to\vec\mu$ transfer property holds for corresponding finite intersections, set differences, connected components, or disjoint unions of such links. 
We easily see  that each region $ \Gamma^{\vec\lambda}_i$ partitions into finitely many pieces, each given as a connected component of a difference of finite intersections of spheres $\rmBdry U_j^{\lambda_j}$. Thus sets obtained by intersecting these with $A$ all inherit the desired  $\vec\lambda\to\vec\mu$ transfer property. 
\par
Since each set $ A\cap \Gamma_i^{\vec\lambda}$ is  compact and subanalytic of dimension $< \dim A$, it has, by our induction on $\dim A$, its own linear isoperimetric inequality.  
Moreover, by the bilipschitz equivalence from Theorem \ref{thm.2} and Proposition \ref{33}, we may now assume 
\vskip.5cm

\noindent{\bf (*)}
{\it \quad \quad The linear isoperimetric inequality is true with the {\bf same constant} $\ \bc\ $ for }
\begin{equation*}
 \ A\cap \Gamma_1^{\vec\lambda}\ ,\ \ldots\ ,\ A\cap \Gamma_\ell^{\vec\lambda} \quad {\bf for\ all}\quad \vec\lambda\in R_{\vec r}^{\vec a}\ .
\end{equation*}
The uniformity of this estimate will allow us to choose and fix an $\ell$ tuple of scaling factors $\vec\lambda=(\lambda_1,\dots ,\lambda_\ell)\in R_{\vec r}^{\vec a}a\ $   depending on the given chain $S_0\in{\bI}_{k+1}(A)$, but still have mass estimates independent of $S_0$.
\par 
As we did in \S 4.1, we again consider some Lebesgue null sets of $\lambda_i$ to avoid to guarantee some desired properties relative to the given chain.%
 \par
 For any fixed $S\in\bI_{k+1}({\bbR^n} )$ and positive $\lambda_1,\dots ,\lambda_i\not\in\Lambda_{\| S\| +\|\partial S\|}$,  we see as before that 
\begin{equation*}
(\| S\|+\|\partial S\|)\left({\rm Bdry} (W_h^{\vec\lambda})\right)= 0\ ,\quad
 S\hel\overline{W}_h^{\vec\lambda}=S\hel W_h^{\vec\lambda}\ ,\quad   (\partial S)\hel\overline{W}_h^{\vec\lambda}=(\partial S)\hel W_h^{\vec\lambda} ,
\end{equation*} 
for $h=1,\dots , i$.  Now we find, in $\cup_{h=1}^i U_h^{\lambda_h}$, the two decompositions
\begin{equation} 
\label{eq.1}
S\hel\left( \cup_{h=1}^i U_h^{\lambda_h}\right) \ =\  \sum_{h=1}^i S\hel W_h^{\vec\lambda}\quad{\rm and}\quad (\partial S)\hel\left( \cup_{h=1}^i U_h^{\lambda_h}\right)  =\  \sum_{h=1}^i (\partial S)\hel W_h^{\vec\lambda}\ ,
\end{equation}
obtained by ignoring the boundaries of the $W_h^{\vec\lambda}$.
\par
Using the Lipschitz map $u_i( x)=r_i^{-1}|x-a_i|$, we have that $U_i^{\lambda}\ =\ \{x\, :\,u_i(x)<\lambda\}$,
and there is another exceptional null set $\Lambda^S_i\subset [0,1]$,   on the complement of which one obtains integral chain slices satisfying the formulas
\begin{equation} 
\label{eq.2}
\begin{split}
\langle S,u_i,\lambda\rangle & = \partial (S\hel U_i^\lambda)- (\partial S)\hel U_i^\lambda \quad\in \bI_k(A \cap \rmBdry U_i^\lambda)\ , \\
\partial\langle S,u_i,\lambda\rangle\ &=\ -\langle  \partial S,u_i,\lambda \rangle   = -\partial [ (\partial S)\hel U_i^{\lambda }] \quad \in \bI_{k-1}(A \cap \rmBdry U_i^\lambda)\ .
\end{split}
\end{equation}
and the corresponding measures satisfy
$$
\| \partial (S\hel U_i^{\lambda})\|\ =\ \|\partial S\|\hel U_i^{\lambda}\ +\ \| \langle S,u_i,\lambda\rangle\|\ ,\quad\
\|\partial [(\partial S)\hel U_i^{\lambda}]\,\|\ =\ \| \langle \partial S,u_i,\lambda\rangle \|\ ,
$$
with $ \|\partial S\|\hel U_i^{\lambda}$ and $ \| \langle S,u_i,\lambda\rangle\|$ being mutually singular.

Moreover, since $\rmLip u_i =  r_i^{-1}$, we have the integral slice mass inequality
\begin{equation}
\label{eq.3}
\int_{\lambda_{\vec r, i}^{\vec a}}^{\mu_{\vec r, i}^{\vec a}}\bM \langle \partial S,u_i,\lambda\rangle\, d\lambda \leq r_i^{-1}\bM( \partial S)\ .
\end{equation}

\noindent We will also need a similar discussion for a fixed $Q\in\bI_{k}(A \cap \Gamma_i^{\vec\lambda})$. If  $j\in \{i+1,\dots ,\ell\}$ and $\lambda_j\not\in\Lambda_{\| Q\|}\ $, then $\ \| Q\|\left({\rmBdry}(W_j^{\vec\lambda})\right)= 0\ $. In case this is true for every such $j\in\{i+1,\dots ,\kappa\}$, one gets the decomposition
\begin{equation}
\label{eq.4}
Q\hel (U_1^{\lambda_1}\cup\dots\cup U_\kappa^{\lambda_\kappa})\ =\ \sum_{j=i+1}^\kappa Q\hel \Gamma_{i,j}^{\vec\lambda}\ . 
\end{equation}
There is also another exceptional null set $\Lambda^Q_i\subset [0,1]$  so that each $\lambda\in [0,1]\setminus \Lambda^Q_i$ gives an integral chain slices  $\langle Q,u_i,\lambda\rangle $  satisfying 
\begin{equation} 
\label{eq.5}
\langle Q,u_i,\lambda\rangle  = \partial (Q\hel U_i^\lambda)- (\partial Q)\hel U_i^\lambda \quad \in \bI_{k-1}(A \cap \Gamma_i^{\vec\lambda})\ , 
\end{equation}
with the orthogonal decomposition of measures
$$
\| \partial (Q\hel U_i^{\lambda})\|\ =\ \|\partial Q\|\hel U_i^{\lambda}\ +\ \| \langle Q,u_i,\lambda\rangle\|\ .
$$

We shall now choose below $\lambda_1,\lambda_2,\dots ,\lambda_\ell$ in order, and use these radii to construct chains 

\noindent$H_1,\dots H_\ell\in\bf I_{k+1}(A)$ so that
\begin{equation*}
 \partial (H_1+\cdots +H_\ell)\ =\ \partial S_0 \quad {\rm and}\quad \bM (H_1)\ , \dots ,\ \bM(H_\ell)\ \leq \bc \bM (\partial S_0)\  ,
\end{equation*}
with $\bc$ depending only on $A$, and independent of $S_0$. This will complete the proof by letting $S=H_1+\cdots +H_\ell$. \ 
\par
For our first choice, we apply (\ref{eq.2}) and the second inequality of (\ref{eq.3}) with $i=1$, $S=S_0$,  to find a ``good'' radius
$$
\lambda_1\in [\lambda_{\vec r, 1}^{\vec a},\mu_{\vec r, 1}^{\vec a} ]\ \setminus\ \left(\Lambda_{\| S_{0}\|+\|\partial S_{0}\| }\cup \Lambda^{S_{0}}_1\right)
$$ 
so that the chain  
\begin{equation*}
R_1\ :=\ \partial \langle S_0,u_1,\lambda_1\rangle\ =\ -\langle\partial S_0,u_1,\lambda_1\rangle\ \ = - \partial [(\partial S_0)\hel U_1^{\lambda_1}]\ 
 \in \bI_{k-1}(A \cap \Gamma_1^{\vec\lambda})
\end{equation*}
has mass satisfying
$$
 \bM(R_1)\ \leq\ [(\mu_{\vec r, 1}^{\vec a} -\lambda_{\vec r, 1}^{\vec a} )r_i]^{-1}\bM(\partial S_0)\ =\ \bc\bM(\partial S_0)\ .
$$
Since the chain $P_1 : = \langle S_0,u_1,\lambda_1\rangle \in\bI_{k}(A\cap \Gamma_1^{\vec\lambda})$ has $\partial P_1 = R_1$,  we may apply our dimension induction (*), with $A$ replaced by 
the lower dimensional set $A\cap \Gamma_1^{\vec\lambda}$ and $S_0$ replaced by  $P_1$, to obtain a chain ${Q_1}\in\bI_{k}(A\cap \Gamma_1^{\vec\lambda})$ with $\partial  Q_1 = \partial  P_1 = R_1$ and
$$
 \bM( Q_1)\ \leq \bc\bM(\partial Q_1)\ =\  \bc \bM (R_1)\ \leq\ \bc\bM(\partial S_0)\ .
$$
Note that $\quad J_1\ :=\ (\partial S_0)\hel W_1^{\vec\lambda}\ +\ Q_1\quad \in \bI_k(A \cap \rmClos U_1^{\lambda_1})$ has
$$
\partial J_1\ =\   \partial[(\partial S_0)\hel W_1^{\vec\lambda}] +\ \partial P_1\ =\ -R_1+R_1\ =\ 0\ ,
$$
and
$$
 \bM(J_1)\ \leq\ \bM(\partial S_0)+ \bM( Q_1) \ \leq\ \bc\bM(\partial S_0)\ .
$$
Inasmuch as the cycle $J_1$ has support  in $\ A\cap\bU  (a_1,r_1)\ $, we infer from  Corollary {\ref{cor.1}} that the contraction 
$$
H_1\ := \  -h_{1\#}\left(\lseg 0,1\rseg\times J_1\right)\ \in \bI_{k+1}(A)\ 
$$
has  $\quad \partial H_1\ = \ J_1\quad$ and 
$$
\bM (H_1) +  \bM(\partial H_1)\ \leq\ \left((\rmLip h_1)^{k+1}+1\right) \bM(J_1)\ \leq\ \bc\bM(\partial S_0)\ .
$$
Letting $S_1 := S_0 - H_1$, we have that
\begin{equation*}
\bM(\partial S_1)\ \leq  \bM(\partial S_0)+\bM(\partial H_1)\ \leq \bc\bM(\partial S_0)\ ,
\end{equation*}
and we easily see that one may replace $S_0$ by $S_1$ to prove the theorem.  The advantage in passing from $S_0$ to $S_1$ is that subtractng $H_1$ essentially ``moves the boundary'' out of the ball $U_1^{\lambda_1}$. In fact, the chain
\begin{equation*}
\partial S_1\ =\ \partial S_0 -\partial H_1\ =\ \partial S_0  - (\partial S_0)\hel W_1^{\vec\lambda}-Q_1\ =\  (\partial S_0)\hel (A\setminus W_1^{\vec\lambda})-Q_1\ 
\end{equation*}
has
$$
\rmspt (\partial S_1)\ \subset\  (A\setminus W_1^{\vec\lambda})\cup  \Gamma_1^{\vec\lambda} \ =\  A\setminus U_1^{\lambda_1}\ .
$$
\vskip.2cm
\par
We wish to continue moving the boundary in steps out of the remaining balls $U_i^{\lambda_i}$. However, to find $H_2$ for the next modification $\ S_2 = S_1-H_2\ $, or generally $H_i$ for $\ S_i=S_{i-1}-H_i$, is somewhat more involved. For example, to get the boundary of $S_2$ out of both balls
$$
U_1^{\lambda_1}\cup U_2^{\lambda_2}\ =\ W_1^{\vec\lambda}\cup \Gamma_{1,2}^{\vec\lambda\ o}\cup W_2^{\vec\lambda}\ ,
$$
we will below need to choose $H_2$ to attach to $H_1$ along the interface $\Gamma_{1,2}^{\vec\lambda\ o}$.  For the reader's convenience, we will first go through this second step, involving the choice of $H_2$, carefully before describing the general $i$th step.
\par
 For the second step, we again start with the choice of radius  $\lambda_2$. By
applying (\ref{eq.2}), (\ref{eq.3}), (\ref{eq.4}), (\ref{eq.5}) with $i=2$, $S=S_1$, and $Q=Q_1$.  We obtain
$$               
\lambda_2\in [\lambda_{\vec r, 2}^{\vec a},\mu_{\vec r, 2}^{\vec a} ]\ \setminus\ \left(\Lambda_{\| S_{0}\|+\|\partial S_{0}\|+\| S_{1}\|+\|\partial S_{1}\|+\| Q_1\| }\cup \Lambda^{S_{0}}_1\cup \Lambda^{S_{1}}_1\cup \Lambda^{Q_{1}}_1\right)
$$ 
so that the chain
\begin{equation*}
R_2\ :=\   \partial \langle S_1,u_2,\lambda_2\rangle\ =\ -\langle\partial S_1,u_2,\lambda_2\rangle\ =  -\partial [(\partial S_1)\hel U_2^{\lambda_2}]
\end{equation*}
has mass satisfying
\begin{equation*}
\bM(R_2)\ \leq\  [(\mu_{\vec r, 2}^{\vec a} -\lambda_{\vec r, 2}^{\vec a} )r_2]^{-1}\bM(\partial S_1)\
 \leq\ \bc\bM(\partial S_1)\ \leq \ \bc\bM(\partial S_0)\  .
\end{equation*}
\par
Since the chain $P_2 : = \langle S_1,u_2,\lambda_2\rangle \in\bI_{k}(A\cap \Gamma_2^{\vec\lambda})$ has $\partial P_2 = R_2$,  we may again apply our dimension induction (*), this time with $A$ replaced by the lower dimensional set $A\cap \Gamma_2^{\vec\lambda}$ and $S_0$ replaced by  $P_2$, to obtain a chain ${Q_2}\in\bI_{k}(A\cap \Gamma_2^{\vec\lambda})$ with $\partial  Q_2 = \partial  P_2 = R_2$ and
$$
 \bM( Q_2)\ \leq \bc\bM(\partial Q_2)\ =\  \bc \bM (R_2)\ \leq\ \bc\bM(\partial S_0)\ .
 $$
 The chain $Q_{1,2}:=Q_1\hel\Gamma_{1,2}^{\vec\lambda} =  Q_1\hel\ \overline{W}_2^{\vec\lambda}$ shows up in the boundary calculation 
\begin{equation*}
\begin{split}
\partial(S_1\hel U_2^{\lambda_2} )\ &=\  \langle S_1,u_2,\lambda_2\rangle\ +\ (\partial S_1)\hel U_2^{\lambda_2}  \\
&=\  P_2\ +\ (\partial S_1)\hel \overline{W}_2^{\vec\lambda} \\
&=\ P_2\ +\ (\partial S_0)\hel \overline{W}_2^{\vec\lambda} \  -\ (\partial H_1)\hel \overline{W}_2^{\vec\lambda}   \\
&=\ P_2\ +\ (\partial S_0)\hel \overline{W}_2^{\vec\lambda} \   -\  (\partial S_0)\hel\ (W_1^{\vec\lambda}\cap \overline{W}_2^{\vec\lambda}) \ -\ Q_1\hel  \overline{W}_2^{\vec\lambda} \\ 
&=\ P_2\ +\ (\partial S_0)\hel\ W_2^{\vec\lambda}\  +0\ -\  Q_{1,2}\ .
\end{split}
\end{equation*}
Thus the chain  $\quad J_2\ :=\ (\partial S_0)\hel\ W_2^{\vec\lambda}+Q_2 - Q_{1,2}\quad$ has 
$$
\partial J_2 \ =\ \partial [ (\partial S_0)\hel\ W_2^{\vec\lambda}+P_2 - Q_{1,2}]\ =\  \partial^2 (S_1\hel  U_2^{\vec\lambda} )\ =\ 0\ ,
$$
 and
$$
 \bM(J_2)\leq\ \bM(\partial S_0)+ \bM( Q_2)+ \bM( Q_1) \ \leq\ \bc\bM(\partial S_0)\ .
$$
Since $J_2$ also has support in $\ A\cap\bU  (a_2,r_2)\ $, we find from Corollary {\ref{cor.1}} that the contraction 
$$
H_2\ := \  -h_{2\#}\left(\lseg 0,1\rseg\times J_2\ \right)\ \in \bI_{k+1}(A)\ 
$$
has $\quad \partial H_2\ = \   J_2\quad$ and 
$$
\bM (H_2) +  \bM(\partial H_2)\ \leq\ \left((\rmLip h_2)^{k+1}+1\right) \bM(J_2)\ \leq\ \bc\bM(\partial S_0)\ .
$$
Letting $S_2 := S_1 - H_2 = S_0-H_1-H_2$\ , we have
\begin{equation*}
\bM(\partial S_2)\ \leq  \bc[\bM(\partial S_1)+\bM(\partial H_2)]\ \leq \bc\bM(\partial S_0)\ ,
\end{equation*}
and
\begin{equation*}
\begin{split}
\partial S_2\ &= \partial S_1 -\partial H_2\ =\  (\partial S_0)\hel (A\setminus W_1^{\vec\lambda})-Q_1\  -Q_2 +Q_{1,2}  -(\partial S_0)\hel W_2^{\vec\lambda}\\
&=\     (\partial S_0)\hel [A\setminus (W_1^{\vec\lambda}\cup W_2^{\vec\lambda})]\ -\ Q_1+Q_{1,2}   -\ Q_2\\
&=\     (\partial S_0)\hel [A\setminus (U_1^{\lambda_1}\cup U_2^{\lambda_2})]\ -\ Q_1+Q_{1,2}   -\ Q_2\ .
\end{split}
\end{equation*}
because $\|\partial S_0)\| (\rmBdry U_1^{\lambda_1}) = 0$.  Inasmuch as
$$
Q_1\hel U_1^{\lambda_1}= 0\ ,\ \  Q_1\hel U_2^{\lambda_2} = Q_{1,2}\ ,\ \  Q_{1,2}\hel U_1^{\lambda_1}=0\ ,\ \ Q_2\hel U_1^{\lambda_1}= 0\ ,\ \   Q_2\hel U_2^{\lambda_2} =  0 \ ,
$$
we see that
\begin{equation*}
\begin{split}
(\partial S_2)\hel (U_1^{\lambda_1}\cup U_2^{\lambda_2} )\ &=\ 0\ -\ (Q_1-Q_{1,2} ) \hel (U_1^{\lambda_1}\cup U_2^{\lambda_2})\  -\ Q_2\hel (U_1^{\lambda_1}\cup U_2^{\lambda_2} )\\
&=\ 0\ -\ (Q_{1,2} - Q_{1,2} )  -\ 0\ =\ 0\ ,
\end{split}
\end{equation*}
and conclude that 
$$
\rmspt (\partial S_2)\ \subset \ A\setminus\  (U_1^{\lambda_1}\cup U_2^{\lambda_2})\ ,
$$
which completes the second step. 
\vskip.2cm

\vskip.2cm

Modifying the above, we now describe how to obtain the $i$th step from the $(i-1)$st.  We assume that  $i\leq\ell-1$ and that we have already chosen,
\noindent for each $h\in \{1,\ \dots ,i-1\ \}$, 

$$
\lambda_h >0\ ,\quad Q_h\in\bI_{k}(A\cap \Gamma_h^{\vec\lambda})\ ,\quad H_h\in\bI_{k+1}(A)\ ,\quad S_{h}:=S_0-\sum_{\eta =1}^{i-1}H_\eta\in\bI_{k+1}(A)\ ,
$$
to satisfy the three conditions:
\vskip.2cm
\hskip.5cm (I)$_h$ $\quad\partial H_h\ = \ (\partial S_0)\hel\ W_h^{\vec\lambda}\ +\ Q_h -\ \sum_{\eta =1}^{h-1}Q_{\eta,h}\ $ where $\ Q_{\eta,h}:= Q_\eta\hel\Gamma_{\eta,h}^{\vec\lambda}$\ ,
\vskip.2cm
\hskip.5cm (II)$_h$ $\quad\bM(Q_h)\ +\ \bM(H_h)\ +\ \bM(\partial H_h)\ +\ \bM(\partial S_h)\ \leq\ \bc\bM(\partial S_0)$\ .
\vskip.2cm
\hskip.5cm (III)$_h$  $\quad\quad\quad\rmspt (\partial S_h)\ \subset \ A\setminus\  (U_1^{\lambda_1}\cup\ \cdots\ \cup U_{h}^{\lambda_{h}}\ )$\ ,
\vskip.3cm
\noindent(Here the formula for $\partial H_1$ is correct with the convention $\sum_{\eta =1}^{0}Q_{\eta,h}=0$.)
\par
\vskip.2cm
We first choose $\lambda_i$ by applying (\ref{eq.2}), (\ref{eq.3}), (\ref{eq.4}), and (\ref{eq.5}),  this time  with $i=i$, $S=S_{i-1}$, and $Q=Q_1,\dots,Q_{i-1}$\ . We obtain
\begin{equation*}
\lambda_i \in [\lambda_{\vec r, i}^{\vec a},\mu_{\vec r, i}^{\vec a} ]\ \setminus\ \Lambda_{\sum_{h=0}^{i-1}\| S_{h}\|+\|\partial S_{h}\|+\|Q_h\| }
\setminus\cup_{h=0}^{i-1}\left(\Lambda^{S_{h}}_1\cup \Lambda^{Q_{h}}_1\right)
\end{equation*}
so that the chain
\begin{equation*}
R_i\ :=\   \partial \langle S_{i-1},u_i,\lambda_i\rangle\ =\ -\langle\partial S_{i-1},u_i,\lambda_i\rangle\ =  -\partial [(\partial S_{i-1})\hel U_i^{\lambda_i}] \in \bI_{k-1}(A \cap \Gamma_i^{\vec\lambda})
\end{equation*}
has mass satisfying
\begin{equation*}
 \bM(R_i)\ \leq\  [(\mu_{\vec r, i}^{\vec a} -\lambda_{\vec r, i}^{\vec a} )r_i]^{-1}\bM(\partial S_{i-1})\ \leq\ \bc\bM(\partial S_{i-1})\ \leq \bc\bM(\partial S_0)\ .
\end{equation*}
Since the chain $P_{i} : = \langle S_{i-1},u_i,\lambda_i\rangle \in\bI_{k-1}(A\cap \Gamma_{i}^{\vec\lambda})$ has $\partial P_i = R_i$,  we may again apply our dimension induction (*), this time with $A$ replaced by the lower dimensional set $A\cap \Gamma_i^{\vec\lambda}$ and with $S_0$ replaced by  $P_i$, to obtain a chain ${Q_i}\in\bI_{k}(A\cap \Gamma_i^{\vec\lambda})$ such that $\ \partial  Q_i = \partial  P_i = R_i\ $ and
$$
\bM(Q_i)\ \leq\ \bc\bM(\partial Q_i)\ =\  \bc \bM (R_i)\ \leq\ \bc\bM(\partial S_0)\ .
 $$
 The chains  \quad$Q_{h,i}: = Q_h\hel  \Gamma_{h,i}^{\vec\lambda}= Q_h\hel\overline{W}_i^{\vec\lambda}$\quad for $h=1,\dots ,i-1$,  show up in the following boundary calculation using (\ref{eq.2}),  (III)$_{i-1}$,  and (II)$_h$,
\begin{equation*}
\begin{split}
\partial(S_{i-1}&\hel U_i^{\lambda_i} )\ =\  \langle S_{i-1},u_i,\lambda_i\rangle\ +\ (\partial S_{i-1})\hel U_i^{\lambda_i} \\
&=\  \langle S_{i-1},u_i,\lambda_i\rangle\ +\ (\partial S_{i-1})\hel \overline{W}_i^{\vec\lambda} \\
&=\ P_i\ +\ (\partial S_0)\hel \overline{W}_i^{\vec\lambda} \  -\  \sum_{h=1}^{i-1} (\partial H_h)\hel \overline{W}_i^{\vec\lambda} \\ 
&=\ P_i\ +\  (\partial S_0)\hel \overline{W}_i^{\vec\lambda} \  -\   \sum_{h=1}^{i-1} \left( (\partial S_0)\hel\ (W_h^{\vec\lambda}\cap \overline{W}_i^{\vec\lambda})\ +\ Q_h\hel  \overline{W}_i^{\vec\lambda}\ -\ \sum_{\eta=1}^h Q_{\eta,h} \hel\overline{W}_i^{\vec\lambda}\right)\\
&=\ P_i \ +\  (\partial S_0)\hel\ W_i^{\vec\lambda}\ -\ \sum_{h=1}^{i-1}\ (\ 0\ +\ Q_{h,i}\ -\ 0\ )\  .
\end{split}
\end{equation*}
Then the chain  $\quad J_i\ =\  (\partial S_0)\hel W_i^{\vec\lambda}\ +\ Q_i\  -\ \sum_{h=1}^{i-1} Q_{h,i} \in \bI_k(A \cap \rmClos U_i^{\lambda_i})\quad$ satisfies
\begin{equation*}
\partial J_i\ =\ \partial \left(\ (\partial S_0)\hel\ W_i^{\vec\lambda}\ +\ P_i\  -\ \sum_{h=1}^{i-1}Q_{h,i} \right)\
=\  \partial^2 (S_{i-1}\hel  U_{i}^{\lambda_i} )\ =\ 0\ 
\end{equation*}
and, by  (II)$_{h}$,
$$
\bM(J_i)\ \leq\ \bM(\partial S_0)\ +\ \bM( Q_i) + \sum_{h=1}^{i-1}  \bM(Q_h) \ \leq\ \bc\bM(\partial S_0)\ .
$$
\noindent Since this cycle  has support in $A\cap\bU (a_i,r_i)$, we see from Corollary \ref{cor.1} that the contraction
$$
H_i\ := \  -h_{i\#} \left[ \lseg 0,1\rseg\times J_i \ \right]\ \in \bI_{k+1}(A)\ 
$$
has $\quad\partial H_i\ = \   J_i \quad$ and  that
$$
\bM (H_i) +  \bM(\partial H_i)\ \leq\  \left( (\rmLip h_i)^{k+1}+1 \right) \bM(J_i) \leq\ \bc\bM(\partial S_0)\ .
$$

Letting $\ S_i :=\ S_{i-1} - H_i\ =\ S_0- \sum_{h=1}^{i} H_h$\ , we have from (II)$_{i-1}$ that
\begin{equation*}
\bM(\partial S_i)\ \leq\  \bc[\bM(\partial S_{i-1})+\bM(\partial H_i)]\ \leq\ \bc\bM(\partial S_0)\ ,
\end{equation*}
and from (I)$_{h}$ that
\begin{equation*}
\begin{split}
\partial S_i\ &= \ \partial S_0\ -\ \sum_{h=1}^{i}\partial H_h\\
&= \ \partial S_0\ -\ \sum_{h=1}^i \left(\ (\partial S_0)\hel W_h^{\vec\lambda})\ +\  Q_h-\sum_{\eta=1}^{h-1}Q_{\eta,h}\ \right)  \\
&= \ (\partial S_0)\hel (A\setminus \cup_{h=1}^{i} W_h^{\vec\lambda})\ -\ \sum_{h=1}^i \left(Q_h-\sum_{\eta=1}^{h-1}Q_{\eta,h}\ \right)  \\
&=\  (\partial S_0)\hel (A\setminus \cup_{h=1}^{i} U_h^{\lambda_h})\ -\ \sum_{h=1}^i Q_h\ +\ \sum_{h=2}^i\sum_{\eta=1}^{h-1}Q_{\eta,h}
\end{split}
\end{equation*}
because $\ \|\partial S_0)\| (\rmBdry U_h^{\lambda_h}) = 0$ for $h=1,\dots , i\ $.  Abbreviating ${\tilde U}_i=\cup_{h=1}^{i} U_h^{\lambda_h}$, we see that  $\ Q_i\hel {\tilde U}_i = 0\ $ because $\Gamma_i^{\vec\lambda}\cap {\tilde U}_i =\ \emptyset$, and that we may decompose $Q_h\hel {\tilde U}_i$ by applying (\ref{eq.4}) with $Q$, $i$, $\kappa$ replaced by  $Q_h$, $h$, $i$. We deduce that
\begin{equation*}
\begin{split}
(\partial S_i)\hel {\tilde U}_i\ &= 
\ (\partial S_0)\hel (A\setminus {\tilde U}_i)\hel {\tilde U}_i\ -\ \sum_{h=1}^{i-1} Q_h\hel {\tilde U}_i\ +\  \sum_{h=2}^{i} \sum_{\eta=1}^{h-1}Q_{\eta,h}\hel {\tilde U}_i\\
&=\ 0\  -\ \sum_{h=1}^{i-1}  \sum_{j=h+1}^{i} Q_{h,j}\ +\  \sum_{h=2}^{i} \sum_{\eta=1}^{h-1}Q_{\eta,h}\quad =\quad 0\ .
\end{split}
\end{equation*}
Thus 
$$
\rmspt (\partial S_i)\ \subset \ A\setminus\ {\tilde U}_i\ =\ A\setminus\  (U_1^{\lambda_1}\cup\ \cdots\ \cup U_{h}^{\lambda_{h}}\ )\ ,
$$
\vskip.2cm
\noindent and we have now verified (I)$_i$, (II)$_i$, and (III)$_i$ and completed the $i$th step for all $i\leq \ell-1$.  In particular,
\begin{equation*}
\rmspt (\partial S_{\ell-1})\ \subset \ A\setminus \cup_{i=1}^{\ell-1}U_i^{\lambda_i}\ \subset\ A\setminus\cup_{i=1}^{\ell-1}\bU(a_i,r_i/2)\ \subset A\cap \bU(a_\ell,r_\ell/2)\ .
\end{equation*}
Finally by defining $\ H_\ell\ : =\ -h_{\ell\#} \left( \lseg 0,1\rseg\times \partial S_{\ell-1}\right)\ $ and $\ S=H_1+\cdots +H_\ell\ $, we find that
$$
\partial H_\ell\ =\  \partial S_{\ell-1} \ =\ \partial S_0\ -\ \sum_{i=1}^{\ell-1}\partial H_i\ ;\quad{\rm hence,}\quad \partial S\ = \partial S_0\ ,
$$
and, by using  (II)$_1$,\dots ,(II)$_{\ell-1}$, that
$$
\bM(S)\ \leq\  \bc\sum_{i=1}^{\ell-1}\bM(\partial H_i) + \bM (H_\ell)\ \leq\ \bc\bM(\partial S_0) +   (\rmLip h_\ell)^k \bM(\partial S_{\ell-1})\ \leq\ \bc\bM(\partial S_0)\ ,
$$
which completes the proof.
\end{proof}



\section{Applications}

\begin{Empty}[Normal Currents]
\label{51}
The exact same proof shows that {\it  if $A \subset \Rn$ is compact and subanalytic, $k=0,1,2,\ldots$ and $S_0 \in \bN_{k+1}A)$ is a normal current of dimension $k+1$ supported in $A$, then there exists $S \in \bN_{k+1}(A)$ such that $\partial S=\partial S_0$ and} $\bM(S) \leq \bc(A) \bM(\partial S)$. This seems to be new even in the case when $A$ is a compact real analytic submanifold of $\Rn$.
\end{Empty}

\begin{Empty}[Other coefficients groups]
\label{52}
We observe that the proof further generalizes to the case of a general normed, complete, Abelian group $G$ of coefficients, \cite{WHI.99.rectifiability}, \cite{DEP.HAR.07}. Here $\calR_k(A;G)$ and $\calF_k(A;G)$ denote the groups consisting of those $k$ dimensional, respectively rectifiable and flat chains with coefficients in $G$, supported in $A$, and 
\begin{equation*}
\begin{split}
\calI_{k+1}(A;G) & = \calR_{k+1}(A;G) \cap \left\{ S : \partial S \in \calR_{k}(A;G) \right\} \\
\calN_k(A;G) & = \calF_k(A;G) \cap \left\{ S : \bM(S) + \bM(\partial S) < \infty \right\}
\end{split}
\end{equation*}
where $\bM$ denotes the usual Euclidean Hausdorff mass of a rectifiable chain in $\Rn$, relaxed to the class of flat chains. {\it If $k=1,2,\ldots$ and $S_0 \in \calI_{k+1}(A;G)$ (resp. $S_0 \in \calN_{k+1}(A;G)$), then there exists $S \in \calI_{k+1}(A;G)$ (resp. $S \in \calN_{k+1}(A;G)$) such that $\partial S=\partial S_0$ and $\bM(S ) \leq \bc(A) \bM(\partial S)$}. This indeed encompasses the previous cases since $\calI_{k+1}(A;\Z)\cong\bI_{k+1}(A)$ and $\calN_{k+1}(A;\R)\cong\bN_{k+1}(A)$.
\par 
We ought to say a word about the case $k=0$ of the proof. Here one replaces the expression $(\partial S_0)\left(\ind_{W_i^{\vec\lambda}}\right)$ using the total multiplicity morphism $\chi : \calF_0(A;G) \to G$, see \cite[4.3.3]{DEP.HAR.07} to find that $g_i = \chi[(\partial S_0) \hel W_i^{\vec\lambda} ] $, while recalling  that $\chi$ is finitely additive and that $|\chi(T)| \leq \calF(T) \leq \bM(T)$.
\end{Empty}

\begin{Empty}[Comparing homology groups] 
\label{53}
We use the same notations as in \ref{52}, and we define the groups of cycles and boundaries
\begin{equation*}
\begin{split}
\bZ_k^\calI(A;G) &= \calI_{k}(A;G) \cap \left\{ T : \partial T = 0 \right\} \\
\bB_k^\calI(A;G) &= \calI_{k}(A;G) \cap \left\{ T : T = \partial S \text{ for some } S \in \calI_{k+1}(A;G) \right\}
\end{split}
\end{equation*}

\noindent as well as the corresponding homology group $\bH^\calI_{k}(A;G) = \bZ_k^\calI(A;G)/\bB_k^\calI(A;G)$. One checks that $\bH^\calI_0(\{0\};G)=G$ and, as in \cite[3.7, 3.9, 3.10, 3.14]{DEP.06.HOM} one shows that the functors $\bH^\calI_k(\cdot;G)$ and $\check{H}_k(\cdot;G)$ (\v{C}ech homology with coefficients in $G$) are naturally equivalent on the category of $(\bH^\calI,k)$ locally connected subsets of Euclidean space and their Lipschitz maps. According to Theorem \ref{thm.1} each compact subanalytic set $A \subset \Rn$ is $(\bH^\calI,k)$ locally connected, recall \cite[3.11]{DEP.06.HOM}. Thus in that case, $\bH^\calI_k(A;G) \cong \check{H}_k(A;G) \cong H_k(A;G)$ where $H_k(A;G)$ denotes singular homology and the last equivalence holds because $A$ is triangulable. 
\end{Empty}

\begin{Empty}[Homology of Normal Currents]
\label{54}
We can repeat the argument made in \ref{53} with the usual normal currents. Letting
\begin{equation*}
\begin{split}
\bZ_k(A) &= \bN_k(A) \cap \left\{ T : \partial T = 0 \right\} \\
\bB_k(A) &= \bN_k(A) \cap \left\{ T : T = \partial S \text{ for some } S \in \bN_{k+1}(A) \right\}
\end{split}
\end{equation*}
and $\bH_k(A) = \bZ_k(A)/\bB_k(A)$ we note that $\bH_0(\{0\}) = \R$ and, referring to \cite[3.14]{DEP.06.HOM} again that $\bH_k(A) \cong \check{H}_k(A;\R) \cong H_k(A;\R)$ in case $A \subset \Rn$ is compact and subanalytic. 
\end{Empty}

\begin{Empty}[Cohomology of Charges]
\label{55}
A complex of cochains on  a compact subset $A \subset \Rn$ is defined and studied in \cite{DEP.HAR.PFE.17}. A {\em charge of degree $k$ on $A$} is a linear functional $\alpha : \bN_k(A) \to \R$ with the following property. For every $\veps > 0$ there exists $\theta > 0$ such that $|\alpha(T)| \leq \theta \bF(T) + \veps \bN(T)$ where $\bF(T)$ is the flat norm of $T$ and $\bN(T)=\bM(T)+\bM(\partial T)$. According to the main result of \cite{DEP.MOO.PFE.08}, {\it a linear functional $\alpha : \bN_k \to \R$ is a charge of degree $k$ if and only if there exist continuous forms $\omega \in C(\Rn,\bigwedge_k\Rn)$ and $\zeta \in C(\Rn,\bigwedge_{k-1}\Rn)$ such that
\begin{equation*}
\alpha(S) = \int_{\Rn} \la \omega, \vec{S} \ra d\|S\| + \int_{\Rn} \la \zeta , \vec{\partial S} \ra d\|\partial S\|
\end{equation*}
whenever} $S \in \bN_k(A)$, in other words $\alpha = \omega + d \zeta$. Let $\bCH^k(A)$ denote the vector space of charges of degree $k$ in $A$. The notions of cocycle and coboundary for charges in $A$ are readily defined in terms of their exterior derivatives $d = \partial^*$:
\begin{equation*}
\begin{split}
\bZ^k(A) &= \bCH^k(A) \cap \left\{ \alpha : d\alpha = 0\right\} \\
\bB^k(A) &= \bCH^k(A) \cap \left\{ d\beta : \beta \in \bCH^{k-1}(A) \right\} \,.
\end{split}
\end{equation*}
This in turn yields a cohomology space $\bH^k(A) = \bZ^k(A)/\bB^k(A)$. Furthermore $\bCH^k(A)$ is given a structure of Banach space with the norm
\begin{equation*}
\|\alpha\| = \sup \left\{ \alpha(S) : S \in \bN_k(A) \text{ and } \bN(S) \leq 1 \right\} \,.
\end{equation*}
\par 
The relevance of the linear isoperimetric inequality \ref{51} in this context is as follows. 
We recall \cite[Chapter 14]{DEP.HAR.PFE.17} that the compact set $A \subset \Rn$ is called {\em $q$-bounded}, $q=0,1,2,\ldots$, whenever the following holds: There exists $\bc(A,q) > 0$ such that for every $T \in \bB_q(A)$ there exists $S \in \bN_{q+1}(A)$ with $\partial S = T$ and $\bM(S) \leq \bc(A,q) \bM(T)$.
It follows from \cite[14.4 and 13.10]{DEP.HAR.PFE.17} that  $A$ is $q$-bounded if and only $\bB^q(A)$ is closed in $\bCH^q(A)$.
In that case $\bH^q(A)$ is a Banach space.
Furthermore, according to \cite[14.9]{DEP.HAR.PFE.17} $\bH^q(A)$ is the strong dual of $\bH_q(A)$ equipped with an appropriate locally convex vector topology, \cite[Chapter 12]{DEP.HAR.PFE.17}. From \S \ref{51} it readily follows that:
\par 
\vskip.2cm
{\it If $A \subset \Rn$ is compact and subanalytic, then  $A$ is $q$-bounded for all $q=0,1,2,\ldots$ and
 $\bH^q(A) \cong \bH_q(A)^* \cong H^q(A;\R)$, the singular cohomology with real coefficients.}
\vskip.2cm
\noindent The latter may be interpreted as a de Rham Theorem in this context.
\end{Empty}

\begin{Empty}[Plateau problem in a homology classes]
\label{56}
Let $G$ be a complete normed Abelian group with norm $|\cdot|$ and assume it satisfies the following extra two conditions.
\begin{enumerate}
\item[(A)] $G \cap \{ g : |g| \leq \kappa \}$ is compact for every $\kappa > 0$;
\item[(B)] $G$ is a \textsc{White} group, i.e. $G$ contains no nonconstant curve of finite length.
\end{enumerate}
We also let $A \subset \Rn$ be a compact subanalytic set, and $k=1,2,\ldots$. Recall the notations of \ref{52}.
Given $T_0 \in \calR_k(A;G)$ with $\partial T_0 = 0$, the following variational problem admits a minimizer:
\begin{equation*}
(\calP) \begin{cases}
\text{minimize } \bM(T)\\
\text{among } T \in \calR_k(A;G) \text{ with } T - T_0 = \partial S \text{ for some } S \in \calR_{k+1}(A;G).
\end{cases}
\end{equation*}
As we show below this is a consequence of our linear isoperimetric inequality and of work of \textsc{B. White}.

\begin{proof}[Existence of solution]
Let $\la T_j \ra_j$ be a minimizing sequence of $(\calP)$. 
For each $j$ there exists $S_j \in \calR_{k+1}(A;G)$ such that $T_j-T_0=\partial S_j$. 
According to the linear isoperimetric inequality \ref{52}, we may choose $\hat{S}_j \in \calI_{k+1}(A;G)$ so that $\partial\hat{S}_j = \partial S_j$ and $\bM(\hat{S}_j)\leq \bc(A)\bM(\partial\hat{S}_j)$.
Thus $T_j-T_0=\partial \hat{S}_j$ and 
$$\
\bM(\hat{S}_j)\ \leq\ \bc(A)\bM(\partial \hat{S}_j ) =\bc(A)\bM(T_j-T_0)\ \leq \bc(A)[ 1+\inf(\calP)+\bM(T_0)]\ <\ \infty\ .
$$
for $j$ sufficiently large.
Since $A$ is compact, the deformation theorem \cite{WHI.99.deformation} together with condition (A) above imply that both sets $\{T_j:j=1,2,\ldots\}$ and $\{\hat{S_j}:j=1,2,\ldots\}$ are totally bounded in the flat norm $\calF$.
Consequently there are integers $j_1 < j_2 < \ldots$ and flat $G$ chains $T \in \calF_k(A;G)$ and $\hat{S} \in \calF_{k+1}(A;G)$ such that $\ \lim_i \calF(T-T_{j_i})=0=\lim_i \calF(\hat{S} - \hat{S}_{j_i})$.
Note that 
$$
T-T_0\ =\ \lim_{i\to\infty}(T_{j_i}-T_0)\ =\ \lim_{i\to\infty}\partial \hat{S}_{j_i}\ =\ \partial \hat{S}\ .
$$
Inasmuch as $\bM$ is lower semicontinuous with respect to $\calF$ convergence, one infers that $\bM(T) < \infty$ and $\bM(\hat{S}) < \infty$. 
It therefore follows from \cite{WHI.99.rectifiability} and condition (B) above that $T \in \calR_k(A;G)$ and $\hat{S} \in \calR_{k+1}(A;G)$.
\end{proof}
\end{Empty}

\begin{Empty}[A linear relative isoperimetric inequality]
\label{57}
Here we work with a pair $B\subset A$ of compact subanalytic subsets of $\Rn$ and verify an isoperimetric inequality generalizing our Main Theorem.  
\par
Assuming that  $U := \Rn\setminus B$ and that   $S$ is a chain of finite mass with support in $A$, we will be interested in the {\bf mass of $S$ in} $U$,
$$
\bM( S\hel U)\ =\ \| S\| (U)\ =\ \| S\| (A\setminus B)\ =\ \bM[S\hel(A\setminus B)\, ]\ ,
$$
rather than the total mass $\bM(S)$. For integral currents, our relative version is:
\par 
\begin{Theorem*} There is a constant $\bc(A,B) > 0$ so that, for every $k\in\{1,2,\ldots\}$ and every $S_0 \in \bI_{k}(A)$, there exists an $S \in \bI_{k}(A)$ satisfying
$$
(\partial S)\hel U\ =\ (\partial S_0)\hel U\quad{\rm and}\quad\bM(S\hel U )\ \leq\ \bc(A,B) \bM[(\partial S)\hel U]\  ,
$$   
\end{Theorem*}
\begin{proof}
Note that this does not follow by simply applying the statement of the Main Theorem to $S\hel U$ because the righthand side of the inequality  
 is missing the term $\bM [\partial (S\hel U)]- \bM [(\partial S)\hel U]$.  Nevertheless, we can obtain the desired $S$ by slightly modifying the arguments in \S 3 and \S  4. Of course, the various constructions and choices, as well as the final chain $S$, will now depend on  $B$ (and $U=\Rn\setminus B$) as well as on $A$ and $S_0$.
 \par
For the new relative version of Theorem \ref{thm.1}, we simply require the extra statement that 
 \vskip.1cm
     {\it if $a\in B$, then {\bf the Lipschitz contraction $h$ preserves} $B\cap K$, that is,} 
 $$
 h\left( [0,1]\times (B\cap K)\right)\ =\ B\cap K\ .
 $$
This extra property has already been essentially treated  in $(A_n)(1)$ of the proof of \cite[Th.2.3.1]{VAL.12}. There, the inductive argument, in \textsc{G. Valette}'s notation, gives, for any finite collection $X_1,X_2,\dots,X_s$ of compact subanalytic subsets of $\Rn$, a single Lipschitz neighborhood contraction $r$ of a neighborhood $U_\varepsilon$ of $x_0$ whose restrictions simultaneously contract the $U_\varepsilon \cap X_j$ . So here we are simply using the two sets $X_1=A$, $X_2=B$  to get the desired contractions $h=r$ of $K=X_1\cap \rmClos ( U_\varepsilon)$ to $a=x_0$ that preserves $B\cap K$.
\par
In the new relative version of Corollary \ref{cor.1}, the fact that $h$ preserves both $K$ and $B\cap K$ implies that $h$ preserves $U\cap K=K\setminus B$, that is, $\ h([0,1)\times (U\cap K))\ =\ U\cap K$. Thus we obtain, by applying $\hel U$ in the proof of Corollary \ref{cor.1}, that
$$
H\hel U\ =\ -h_\#\left( [0,1]\times J\right)\hel U\ =\ - h_\#\left( [0,1]\times ( J\hel U)\ \right)\ ,
$$
which gives the additional conclusion 
\begin{equation}
\label{eq.6}
\bM (H\hel U)  \leq ( \rmLip h)^{k+1}{\bM}(J\hel U)\ , 
\end{equation}                                                                                                                                                                                                                                                                                                                                                                                                  which is essentially the local case of our linear relative isoperimetric inequality.   

If the point  $a\not\in B$ and $K$ is small enough so that $K\cap B=\emptyset$, we can still use the contraction $h$ as before and inequality (\ref{eq.6}) remains true because $J\hel U= J$ and $H\hel U=H$.
 \par
For the new version of Theorem \ref{thm.2}, we need the  extra property that the resulting bilipschitz equivalence of $A$ links also preserves the $B$ links, ie. for $\vec\lambda,\vec\mu$ in $R_{\vec r}^{\vec a}$, 
$$
B\cap \rmBdry \bU(a_i,\lambda r_i)\ =\ B\cap L_{\lambda r_i}^{a_i}\ \to\ B\cap \rmBdry \bU(a_i,\mu r_i)\ =\ B\cap L_{\mu r_i}^{a_i}\ .
$$ 
Our present proof will give this property if we simply add the requirement  that our stratification $\calS$ also be compatible with the set  $\R\times B$. 
\par
We now repeat all the constructions of \S 4, look at the restrictions to $U$, and estimate the masses in $U$ in terms of ${\bM}[(\partial S_0)\hel U]$ rather than ${\bM}(\partial S_0)$. In these estimates  we will use the symbol $\bc$ to abbreviate a constant $\bc(A,B)$, depending only on $A$ and $B$. 

When we first employ compactness to find the finite collection of balls ${\bU}(a_i,r_i/2)$ covering $A$, we may choose just from those balls ${\bU}(a,r/2)$ where the center $a\in A$ and {\bf either $a\in B$ or} ${\bB}(a,r)\cap B=\emptyset$.  Thus, when we repeat the constructions  with the resulting $U_i^{\lambda_i}=\bU(a_i,\lambda_ir_i)$, we can use either the new or old version of Theorem \ref{thm.2}, and Corollary \ref{cor.1} depending on whether $a_i\in B$ or $a_i\not\in B$.

We again argue by induction on $\rmdim A $. As before, the bilipschitz equivalence discussion gives us the assumption  (*)  on the inductive validity of the new relative theorem with the same constant when $A,B$ is replaced by every pair $A\cap\Gamma_i^{\vec\lambda},\ B\cap\Gamma_i^{\vec\lambda}$.
\par
For $\lambda\in\bbR\setminus\Lambda_i^S$,  the slices of $S$ and $\partial S$ by $u_i$ are given by integration, and so they  may be restricted to the open set $U$. That is, these slices all commute with the operation $\hel U$. In particular, simply applying $\hel U$ to every term in (\ref{eq.2}), (\ref{eq.3}), and  (\ref{eq.5}) results in  
 \begin{equation*}
\begin{split}
\langle  S\hel U,u_i,\lambda \rangle\   =\  \langle  S,u_i,\lambda \rangle\hel U\   &=\  [\partial (S\hel U_i^\lambda)]\hel U\ - (\partial S)\hel ( U_i^\lambda \cap U )\  ,\\   \langle  (\partial S)\hel U,u_i,\lambda \rangle\   =\ \langle  \partial S,u_i,\lambda \rangle\hel U\   &=\  \partial [ (\partial S)\hel U_i^\lambda ]\hel U\  ,\\
\int_{\lambda_{\vec r, i}^{\vec a}}^{\mu_{\vec r, i}^{\vec a}} \bM(\langle S, u_i, \lambda\rangle\hel U)\, d\lambda\ \ &\leq\ r_i^{-1}\bM( S\hel U)\\
\int_{\lambda_{\vec r, i}^{\vec a}}^{\mu_{\vec r, i}^{\vec a}} \bM(\langle \partial S, u_i, \lambda\rangle \hel U)\, d\lambda \ &\leq\ r_i^{-1}\bM( (\partial S)\hel U)\ .\\
\langle  Q\hel U,u_i,\lambda \rangle\   =\ \langle  Q,u_i,\lambda \rangle\hel U\   &=\  [\partial (Q\hel U_i^\lambda)]\hel U\ - (\partial Q)\hel ( U_i^\lambda \cap U )\  ,
\end{split}
\end{equation*}

\noindent In the $i$th step of the new proof, we find $\lambda_i$ so that the slice at  
$R_i\hel U\ :=\   \partial \langle S_{i-1},u_i,\lambda_i\rangle\hel U$, has mass
\begin{equation*}
 \bM(R_i\hel U)\  \leq\ \bc\bM[(\partial S_{i-1})\hel U]\ \leq\ \bc\bM[(\partial S_0)\hel U]\ .
\end{equation*}
Now the induction on dimension allows us to to replace $P_i= \langle S_{i-1},u_i,\lambda_i\rangle$ by a chain ${Q_i}\in\bI_{k-1}(A\cap \Gamma_i^{\vec\lambda})$ such that $\partial  Q_i = \partial  P_i = R_i$ and
$$
\bM(Q_i\hel U)\ \leq\ \bc\bM[(\partial Q_i)\hel U]\ =\  \bc \bM (R_i\hel U)\ \leq\ \bc\bM[(\partial S_0)\hel U]\ . 
$$
By defining $Q_{\eta,h}$ and  $H_i$ exactly as before, based on  $Q_1,\dots , Q_i$, $S_0$, and contractions $h_i$, we now find mass-in-$U$ estimates
$$
\bM (H_i\hel U) +  \bM[(\partial H_i)\hel U]\ \leq\  \bc \left( \bM[(\partial S_0)\hel U]\ +\ \sum_{h=1}^{i} 
 \bM(Q_h\hel U) \right)\ \leq\ \bc\bM[(\partial S_0)\hel U]\ ,
$$
since the $h_i$ preserves $U\cap K_i$. With the last modification $H_\ell$ also defined exactly as before, we similarly verify $\bM (H_i\hel U)\leq \bc\bM[(\partial S_0)\hel U]\ $. It follows, as before, that $S=H_1+\cdots +H_\ell$ satisfies the relative theorem.
\end{proof}
\end{Empty}

\begin{Empty}[Remark]
\label{58}
Again both the statement and proof of Theorem \ref{57} carry over to chains with coefficients in a complete normed abelian group $G$. The application below in \S \ref{59} generalizes \S \ref{56} and so involves the rectifiability and the group assumptions of \S \ref{56}(A)(B). But  we will no longer be assuming that all chains and their boundaries have finite mass or are rectifiable everywhere. Noting that the statement and proof of Theorem \ref{57} involve the behavior of the chains and their masses only in $U$, we see that we can further generalize Theorem \ref{57} to flat chains $T$ where both $T$ and $\partial T$ are rectifiable with finite {\bf mass in} $U$. Here a chain $T\in\calF_k(A;G)$ is {\bf rectifiable with finite mass in} $U$ provided that 
$$
T\hel U_i\in\calR_k(A;G)\quad  {\rm and}\quad \sup_i\bM(T\hel U_i) < \infty\ ,
$$
for some open sets $U_1\subset U_2 \subset \cdots $ with $\cup_{i=1}^\infty U_i = U$. In this case\  $\la T\hel U_i\ra_i $\  is $\bM$ Cauchy, and we let $T\hel U$ denote the $\bM$ limit.  This limit is well-defined independently of the choice of open sets, is rectifiable, and  coincides with the usual definition of $T\hel U$ in case  $\bM(T)<\infty$. Also one checks  that
$\rmspt (T-T\hel U) \subset B$ \cite[\S5.5]{DEP.HAR.07}.
\end{Empty}

\begin{Empty}[A relative homology Plateau problem]
\label{59}
  First recall that, for rectifiable chains, the  groups of relative cycles, relative boundaries, and relative homology can be defined:
\begin{equation*}
\begin{split}
\bZ_k^\calR(A,B;G) &= \calR_k(A;G) \cap \left\{ T : {\rmspt}(\partial T) \subset B \right\} \\
\bB_k^\calR(A,B;G) &= \calR_k(A;G) \cap \left\{ T : {\rmspt}(T - \partial S)\subset B\text{ for some } S \in \calR_{k+1}(A;G) \right\} \\
\bH^\calR_k(A,B;G) &= \bZ_k^\calR(A,B;G)/\bB_k^\calR(A,B;G)\ .
\end{split}
\end{equation*}
\par
Here we  discuss how the relative isoperimetric inequality of \S \ref{57} is useful for the following relative homology Plateau problem that  generalizes 
\S \ref{56}.

\hskip.2cm
Given $T_0 \in \calR_k(A;G)$ where $B\subset A$ are compact subanalytic subsets of $\Rn$ and  $G$ satisfies \S \ref{56}(A)(B), consider the problem:
\begin{equation*}
(\calP_B) \begin{cases}
\text{minimize } \bM(T)\\
\text{among } T \in \calR_k(A;G) \text{ with } \rm{spt}(T - T_0-\partial S)\subset B  \text{ for some } S \in \calR_{k+1}(A;G).
\end{cases}
\end{equation*}
\par
Note that in case $\rmspt (\partial T_0)\subset B$, i.e. $T_0\in\bZ_k^\calR(A,B;G)$,  one is minimizing mass in the relative homology class
$$
[T_0]\ := \ \bZ_k^\calR(A,B;G)\cap\left\{ T\ :\ T-T_0\in \bB_k^\calR(A,B;G) \right\}\ \in\ \bH^\calR_k(A,B;G)\ .
$$
\begin{proof}[Existence of solution] 
Since $T_0$ is admissible, one easily obtains a mass minimizing sequence $\la T_j \ra_j$ in $\calR_k(A;G)$ and  $\la S_j\ra_j$ in $\calR_{k+1}(A;G)$ with $\rmspt (T_j-T_0-\partial S_j)\subset B$.  The sequence $\bM(T_j)$ has a finite upper bound which we may assume is $\bM(T_0)$.   The chain $T\hel U$ is admissible whenever $T$ is because $\rmspt (T-T\hel U)\subset B$. Also  we may assume $T_0=T_0\hel U$ because  $T_0\hel U$ gives the same admissible class as $T_0$ does .

Even though  the sequence $\bM(\partial T_j)$ is not obviously bounded above,  the equation $(\partial T_j)\hel U=(\partial T_0)\hel U$ gives a bound on the mass in $U$ of $\partial T_j$.  While there is no bound for the mass in $U$ of $S_j$,  the equation  $(\partial S_j)\hel U=T_j\hel U\ -\ T_0\hel U$ shows that   $\partial S_j$ has bounded  mass in $U$ and is rectifiable in $U$. Just like in \S \ref{56}, we can now use Theorem \ref{57}, with $S_0=S_j$, and Remark \ref{58} to replace  $S_j$ by another chain $\hat{S}_j $, with the same boundary in $U$, to assure that the sequence $\hat{S}_j$ has  bounded mass in $U$. Specifically,
$$
\bM(\hat{S}_j\hel U)\leq \bc\bM((\partial S_j)\hel U)=\bc\bM[(T_j-T_0)\hel U]\leq 2\bc\bM(T_0)\  .
$$
To construct the desired rectifiable chains $S\in\calR_{k+1}(A;G)$ and $T\in\calR_{k}(A;G)$ so that $\rm{spt}(T - T_0-\partial S)\subset B$ and $T$ is a mass minimizer for $(\calP_B) $, we will take limits inside of $U$ away from $B$ and then use a diagonal argument.  Accordingly, we define, for $\delta>0$,  $U_\delta :=\ \{ x\in U\  :\ u(x)>\delta\}$ where $u(x)=\rm{dist}(x,B)$. Inasmuch as 
\begin{equation*}
\begin{split}
\int_0^\infty\liminf_{j\to\infty}&[ \bM\la {\hat S}_j,u,t\ra + \bM\la T_j,u,t\ra ]\,dt\
\leq\ \liminf_{j\to\infty}\int_0^\infty[ \bM\la {\hat S}_j,u,t\ra + \bM\la T_j,u,t\ra ]\,dt\\
&\leq\ \sup_j\bM({\hat S}_j\hel U)+\sup_j\bM(T_j\hel U)\ \leq\ (2\bc+1)\bM(T_0)\ <\ \infty\ ,
\end{split}
\end{equation*} 
we can choose a sequence $t_i\downarrow 0$  so that, for all  $i$,  
$$
\liminf_{j\to\infty}[ \bM\la {\hat S}_j,u,t_i\ra + \bM\la T_j,u,t_i \ra ]\ <\ \infty\ ,\quad \bM\la {T}_0,u,t_i \ra\  <\ \infty,\quad \bM\la \partial T_0,u,t_i\ra\ <\ \infty\ .
$$
We can also insist that, for all $i$ and $j$, 
\begin{equation*}
\begin{split}
  \la {\hat S}_j,u,t_i \ra\in \calR_k(A;G)\ ,&\quad \la T_j,u,t_i\ra \in \calR_{k-1}(A;G)\ ,\\
\partial ({\hat S}_j\hel U_{t_i} )  = ( \partial {\hat S}_j)\hel U_{t_i} + \la {\hat S}_j,u,t_i\ra\ ,&\quad \partial ( {T}_j\hel U_{t_i} )  = ( \partial {T}_j)\hel U_{t_i} + \la {T}_j,u,t_i\ra\  .
\end{split}
\end{equation*} 
Inasmuch as  $(\partial {\hat S}_j)\hel U=(T_j-T_0)\hel U$ and $(\partial T_j)\hel U=(\partial T_0)\hel U$, we also have, for every $i=1,2,\dots$, that
\begin{equation*}
\begin{split}
&\liminf_{j\to\infty}\left[ \bM({\hat S}_j\hel U_{t_i} ) + \bM \partial ({\hat S}_j \hel U_{t_i} ) +\bM(T_j\hel U_{t_i})+\bM\partial (T_j\hel U_{t_i}) \right]\\
&\leq \sup_j \left( \bM({\hat S}_j\hel U) + \bM[(\partial {\hat S}_j)\hel U] +   \bM(T_j\hel U) + \bM[(\partial T_j)\hel U ]\right)\\
&\hskip2in +\ \liminf_{j\to\infty}\left[ \bM\la {\hat S}_j,u,t_i\ra + \bM\la T_j,u,t_i \ra \right]\\
&\leq  (2\bc +2+1)\bM(T_0) + \bM[(\partial T_0)\hel U] +\ \liminf_{j\to\infty}\left[ \bM\la {\hat S}_j,u,t_i\ra + \bM\la T_j,u,t_i \ra \right]\ <\ \infty\ .
\end{split}
\end{equation*}

We find a subsequence $\ \la i^{(1)}_j\ra_j \ $ of $\ \la j\ra_j $ giving, as $j\to\infty$,  flat convergences 
$$ 
{\hat S}_{i^{(1)}_j }\hel U_{t_1}\to {\hat S}^{(1)} \in\calR_{k+1}(A;G)\quad {\rm and}\quad  T_{i^{(1)}_j }\hel U_{t_1}\to T^{(1)}\in\calR_k(A;G)\  .
$$ 
For $m=2,3,\dots$, we similarly inductively find subsequences $\ \la i^{(m)}_j\ra_j\ $ of $\ \la i^{(m-1)}_j\ra_j\ $ and, 

\noindent as $j\to\infty$,  flat convergences 
$$
{\hat S}_{i^{(m)}_j }\hel U_{t_m}\to {\hat S}^{(m)} \in\calR_{k+1}(A;G)\quad {\rm and}\quad \ T_{i^{(m)}_j }\hel U_{t_m}\to T^{(m)}\in\calR_k(U;G) \ .
$$
The lower semicontinuity of $\bM$ implies that $\bM({\hat S}^{(m)})\leq 2\bc\bM(T_0)$ and $\bM({\hat T}^{(m)})\leq \bM(T_0)$.  
\par
For $\ell < m$, $U_{t_\ell}\subset U_{t_m}$, and one has ${\hat S}^{(\ell)}= {\hat S}^{(m)}\hel U_{t_\ell}$ and ${\hat T}^{(\ell)}= {\hat T}^{(m)}\hel U_{t_\ell}$. It follows that the sequences $ {\hat S}^{(m)}$ and ${\hat T}^{(m)}$ are $\bM$-Cauchy and $\bM$-convergent to chains 
$S\in\calR_{k+1}(A;G)$ and  $T\in\calR_{k}(A;G)$, respectively, characterized by having $S\hel U_{t_m}={\hat S}^{(m)}$ and $T\hel U_{t_m}=T^{(m)}$ for every $m\in\{1,2,\dots\}$. Taking the diagonal subsequence $\ \la j'\ra_j\ =\ \la i^{(j)}_j\ra_j\ $, one now has, for all $m$, the flat convergences 
$$
\lim_{j\to\infty}{\hat S}_{j'}\hel U_{t_m}={\hat S}^{(m)}=S\hel U_{t_m}\quad{\rm and} \quad\lim_{j\to\infty} T_{j'}\hel U_{t_m}={\hat T}^{(m)}=T\hel U_{t_m}\  .
$$ 
\par
To verify the boundary relation that $R:= T-T_0-\partial S$ has support in $B$,  it suffices to show that $R\hel U_t = 0$ for a.e. $t>0$.  For each $m\in\{1,2,\dots\}$, a.e. $t\in [t_{m+1},t_m]$, and $j'> m$, we have that
\begin{equation*}
\begin{split}
R\hel U_t\ &=\ (T-T_0-\partial S)\hel U_t\\ 
&=\ (T_{j'}-T_0-\partial {\hat S}_{j'})\hel U_t+(T-T_{j'})\hel U_t-(\partial S-\partial {\hat S}_{j'})\hel U_t\ .
\end{split}
\end{equation*}
Taking the flat norm, integrating, and applying \cite{DEP.HAR.07}[Thm.5.2.3(2)] gives
\begin{equation*}
\begin{split}
\int_{t_{m+1}}^{t_m}\calF( R\hel U_t)\,dt\ &=\  \int_{t_{m+1}}^{t_m}\calF[(T-T_0-\partial S)\hel U_t]\,dt\\  
&\leq\ \int_{t_{m+1}}^{t_m}\left( 0 + \calF[(T-T_{j'})\hel U_t] + \calF [ (\partial S-\partial{\hat S}_{j'})\hel U_t)]\right)dt \\
&\leq\ (t_m-t_{m+1}+1)\left[\calF(T-T_{j'} ) + \calF \partial (S-{\hat S}_{j'}) \right]\\
&\leq\ (t_m-t_{m+1}+1)\left[\calF(T-T_{j'} ) +\calF(S-{\hat S}_{j'} ) \right]\ \to\ 0\ {\rm as}\ j\to\infty\ .
\end{split}
\end{equation*}
We conclude that $ (T-T_0-\partial S)\hel U_t =0$ for a.a. positive $t$, and $\rmspt (T-T_0-\partial S)\subset B$. 
\par
Since $T$ is thus admissible for $(\calP_B)$ and $\la T_j\ra_j$ is a mass minimizing sequence,

\noindent$\bM(T)\geq {\calM} :=\lim_{j\to\infty}\bM(T_j)$\ . \ \ To get that the reverse inequality, we may for any $\varepsilon > 0$, choose an $m$ sufficiently large so that $\bM(T^{(m)}) > \bM(T)-\varepsilon$. We may combine this with the mass lower semicontinuity, under the flat convergence of $T_{i^{(m)}_j}\hel U_{t_m}$ to $T^{(m)}$,
to deduce that 
$$
\bM(T)-\varepsilon\ <\ \bM(T^{(m)})\ \leq\ \liminf_{j\to\infty}\bM(T_{i^{(m)}_j}\hel U_{t_m})\ \leq\ \liminf_{j\to\infty}\bM(T_{i^{(m)}_j})\ =\ \calM\ ,
$$
Letting $\varepsilon\downarrow 0$ gives  $\ \bM(T)\leq\calM$, showing that $T$ is the desired mass minimizer for $(\calP_B) $.
\end{proof}
\end{Empty}

\begin{Empty}[A Poincar\'e inequality]
\label{60}
{\it Suppose the set $M$ of regular points of  a compact subanalytic subset  $A$ of dimension $k$ of $\Rn$ is connected and orientable.  There is a  constant $\bc(A)$ so that for any $f\in{\bf BV}(M)$ and any $m\in\bbR$ satisfying  $\ \mathscr{H}^k\{x\, :\, f(x) < m\} =\ \mathscr{H}^k\{x\, :\,  f(x) > m\}\ $,
(i.e. $m$ is a median of $f$\ ), one has the inequalities}
\begin{align*}
&(1)\quad \int_M |f - m|\, d{\mathscr H}^k\ \leq\ \bc(A) \int_M \| Df\|\  \quad {\rm and}\\
&(2)\quad \int_M |f - \overline f|\, d{\mathscr H}^k\ \leq\ \bc(A) \int_M \| Df\| \quad{\rm where}\quad \overline{f}\ =\ {\mathscr H}^k(A)^{-1}\int_M f\, d{\mathscr H}^k\ .
\end{align*}
\begin{proof} 
Since $A\setminus M$ is subanalytic of dimension $< k$, we may assume $\overline M = A$. For (1) we may subtract the constant function $m$ to assume that $m=0$. Thus  the two sets $M_- = \{ x\in M\, :\, f(x)< 0\} $ and $M_+= \{ x\in M\, :\, f(x)> 0\} $ have the same ${\mathscr H}^k$ measure, which is finite because $ {\mathscr H}^k(M)= {\mathscr H}^k(A) <\infty$. Fix an orientation for $M$, and let $\lseg M\rseg$ denote the corresponding $k$ dimensional rectifiable current. Here, the flat chain $\partial\lseg M\rseg$ is also rectifiable because $\rmspt \partial \lseg M\rseg\subset B:=\overline M\setminus M$ and the constancy theorem \cite[4.1.31]{GMT} may be applied to the $k-1$ dimensional strata. From \cite[4.5.9(12)]{GMT}, we see that for almost all $s>0$, the chain $\lseg M\rseg_s :=\lseg M\rseg\hel \{ x\in M\, :\, f(x)> s\}$ is rectifiable and of finite mass and finite boundary mass in $U:= \Rn\setminus B$. 

Applying the  linear relative isoperimetric inequality, Theorem \ref{57}, with $S_0= \lseg M\rseg_s$ we find an $S_s \in \bI_{k}(A)$ satisfying
$$
(\partial S_s)\hel M\ =\ (\partial \lseg M\rseg_s)\hel M\quad{\rm and}\quad\bM(S_s\hel M )\ \leq\ \bc(A) \bM[(\partial \lseg M\rseg_s\hel M]\  ,
$$
Inasmuch as $\partial (S_s-\lseg M\rseg_s)\hel M =0$, the constancy theorem gives an integer $j$ so that 
$$
(S_s-\lseg M\rseg_s)]\hel M\ = j\lseg M\rseg\quad {\rm and}\quad S_s\hel M = (j-1)\left(\lseg M\rseg\ - \lseg M\rseg_s\right) + j\lseg M\rseg_s\ .
$$
Since $\bM(\lseg M\rseg_s)\leq {\mathscr H}^k(M_+) = \frac 12{\mathscr H}^k(M)=\frac 12\bM(\lseg M\rseg)$, we deduce that
$$
\bM(\lseg M\rseg_s)\ \leq\ \bM(S_s\hel M )\ \leq\ \bc(A) \bM\left(\partial \lseg M\rseg_s\hel M\right)\  .
$$ 
We now use the BV coarea formula as in \cite[4.5.9(13)]{GMT} to see that
\begin{align*}
\int_{M_+} |f |\, d{\mathscr H}^k\ &=\ \int_0^\infty{\mathscr H}^k\{ f>s\}\,ds\ =\ \int_0^\infty \bM(\lseg M\rseg_s)\,ds\\ 
&\leq\ \bc(A)\int_0^\infty \bM\left(\partial \lseg M\rseg_s)\hel M\right)\ \,ds\ =\ \bc(A) \int_{M_+} \| Df\|\ .
\end{align*}
The same argument applied to $-f$ gives
$$
 \int_{M_-} |f |\, d{\mathscr H}^k\ \leq\ \bc(A) \int_{M_-} \| Df\|\ ,
$$
and the conclusion of (1)
\begin{align*}
\int_M |f |\, d{\mathscr H}^k\ &= \int_{M_-} |f |\, d{\mathscr H}^k  +\ \int_{M_+} |f |\, d{\mathscr H}^k\\
& \leq\ \bc(A) \int_{M_-} \| Df\|\ + \bc(A) \int_{M_+} \| Df\|\ \leq \bc(A) \int_M \| Df\|\ .
\end{align*}

Conclusion (2) easily follows from (1) because
$$
 \int_M |f - \overline f|\, d{\mathscr H}^k\ \leq\  \int_M |f - m|\, d{\mathscr H}^k\ +\  \int_M |m-\overline f |\, d{\mathscr H}^k\ ,
$$
and
$$
 |m-\overline f |\ =\ {\mathscr H}^k(A)^{-1}\big|\int_M (m\ -\ f)\, d{\mathscr H}^k\big|\ \leq\   {\mathscr H}^k(A)^{-1}\int_M |f - m|\, d{\mathscr H}^k\ ,
$$

\end{proof}
\end{Empty}


\bibliographystyle{amsplain}
\bibliography{LIP-SA-submit2.bbl}

\end{document}